\documentclass[12p]{amsart}
\usepackage{amssymb}
\usepackage{amsmath}
\usepackage{amsfonts}
\usepackage{geometry}
\usepackage{graphicx}
\usepackage{mathrsfs,amssymb}

\usepackage{hyperref}
\usepackage{cleveref}

\theoremstyle{plain}

\newtheorem{definition}{Definition}

\newtheorem{lemma}{Lemma}

\newtheorem{remark}{Remark}

\newtheorem{theorem}{Theorem}
\numberwithin{equation}{section}

\begin{document}
\title[]{The $L^{2}$ sequential convergence of a solution to the mass-critical NLS above the ground state}

\author{Benjamin Dodson}

\begin{abstract}
In this paper we generalize a weak sequential result of \cite{fan20182} to a non-scattering solutions in dimension $d \geq 2$. No symmetry assumptions are required for the initial data. We build on a previous result of \cite{dodson20202} for one dimension.

\end{abstract}
\maketitle

\section{Introduction}
The mass-critical nonlinear Schr{\"o}dinger equation (NLS) is given by
\begin{equation}\label{1.1}
i u_{t} + \Delta u = \mu |u|^{\frac{4}{d}} u = \mu F(u), \qquad u(0,x) = u_{0}, \qquad u : I \times \mathbb{R}^{d} \rightarrow \mathbb{C}, \qquad \mu = \pm 1,
\end{equation}
where $I \subset \mathbb{R}$ is an open interval with $0 \in I$. The case when $\mu = +1$ is the defocusing case, and the case when $\mu = -1$ is the focusing case.

If $u$ solves $(\ref{1.1})$, then for any $\lambda > 0$,
\begin{equation}\label{1.2}
\lambda^{d/2} u(\lambda^{2} t, \lambda x),
\end{equation}
also solves $(\ref{1.1})$ with initial data $\lambda^{d/2} u_{0}(\lambda x)$. The $L^{2}$ norm, or mass, is preserved under $(\ref{1.2})$. Thus, $(\ref{1.1})$ is called $L^{2}$ or mass critical. The $L^{2}$ norm, or mass, is also conserved by the flow of $(\ref{1.1})$. If $u$ is a solution to $(\ref{1.1})$ on some interval $I \subset \mathbb{R}$, $0 \in I$, then for any $t \in I$,
\begin{equation}\label{1.3}
M(u(t)) = \int |u(t,x)|^{2} dx = \int |u(0,x)|^{2} dx.
\end{equation}

It is well-known that the local well-posedness of $(\ref{1.1})$ is completely determined by $L^{2}$-regularity. In the positive direction,  \cite{cazenave1989some}, \cite{cazenave1990cauchy} proved that $(\ref{1.1})$ is locally well-posed on some open interval for initial data $u_{0} \in L^{2}(\mathbb{R}^{d})$. Furthermore, if $u_{0} \in H_{x}^{s}(\mathbb{R}^{d})$ for some $s > 0$,  \cite{cazenave1989some}, \cite{cazenave1990cauchy} proved that $(\ref{1.1})$ was locally well-posed on an open interval $(-T, T)$, where $T(\| u_{0} \|_{H^{s}}) > 0$ depends only on the size of the initial data. Finally,  \cite{cazenave1989some}, \cite{cazenave1990cauchy} proved that there exists $\epsilon_{0} > 0$ such that if $\| u_{0} \|_{L^{2}} < \epsilon_{0}$, then $(\ref{1.1})$ is globally well-posed and scattering.

\begin{definition}[Scattering]\label{d1.1}
A solution to $(\ref{1.1})$ that is global forward in time, that is $u$ exists on $[0, \infty)$, is said to scatter forward in time if there exists $u_{+} \in L^{2}(\mathbb{R}^{d})$ such that
\begin{equation}\label{1.4}
\lim_{t \nearrow \infty} \| u(t) - e^{it \Delta} u_{+} \|_{L^{2}(\mathbb{R}^{d})} = 0.
\end{equation}
A solution to $(\ref{1.1})$ that is global backward in time is said to scatter backward in time if there exists $u_{-} \in L^{2}(\mathbb{R}^{d})$ such that
\begin{equation}\label{1.5}
\lim_{t \searrow -\infty} \| u(t) - e^{it \Delta} u_{-} \|_{L^{2}(\mathbb{R}^{d})} = 0.
\end{equation}
Equation $(\ref{1.1})$ is scattering for any $u_{0} \in L^{2}(\mathbb{R}^{d})$, or for $u_{0}$ in a specified subset of $L^{2}(\mathbb{R}^{d})$, if for any $u_{0} \in L^{2}(\mathbb{R}^{d})$ or the specified subset of $L^{2}(\mathbb{R}^{d})$, there exist $(u_{-}, u_{+}) \in L^{2}(\mathbb{R}^{d}) \times L^{2}(\mathbb{R}^{d})$ such that $(\ref{1.4})$ and $(\ref{1.5})$ hold, and additionally, $u_{-}$ and $u_{+}$ depend continuously on $u_{0}$.
\end{definition}
\noindent In the negative direction, \cite{christ2003asymptotics} showed that local well-posedness fails for $u_{0} \in H^{s}$, $s < 0$.

The qualitative global behavior for $(\ref{1.1})$ in the defocusing case $(\mu = +1)$ has now been completely worked out. A solution to $(\ref{1.1})$ has the conserved quantities mass, $(\ref{1.3})$, energy,
\begin{equation}\label{1.7}
E(u(t)) = \frac{1}{2} \int |u_{x}(t,x)|^{2} dx + \frac{\mu d}{2d + 4} \int |u(t,x)|^{\frac{2d + 4}{d}} dx = E(u(0)),
\end{equation}
and momentum
\begin{equation}\label{1.8}
P(u(t)) = Im \int \nabla u(t,x) \overline{u(t,x)} dx = P(u(0)).
\end{equation}
When $\mu = +1$, $(\ref{1.7})$ is positive definite, so if $u_{0} \in H^{1}(\mathbb{R}^{d})$, then the energy gives an upper bound on $\| u(t) \|_{H^{1}}$ for any $t \in I$. Since $(\ref{1.1})$ is locally well-posed on an interval $[-T, T]$, where $T(\| u_{0} \|_{H^{1}}) > 0$, conservation of energy implies that the local well-posedness result of \cite{cazenave1989some}, \cite{cazenave1990cauchy} can be iterated to a global well-posedness result. Later, $(\ref{1.1})$ was proved to be globally well-posed and scattering for any initial data in $u_{0} \in L^{2}(\mathbb{R}^{d})$ when $\mu = +1$, see \cite{dodson2016global2}, \cite{dodson2016global}, and \cite{dodson2012global}.

In the focusing case $(\mu = -1)$, the existence of non-scattering solutions to $(\ref{1.1})$ has been known for a long time, see \cite{glassey1977blowing}. Let $Q(x)$ be the unique, positive, radial solution of the elliptic partial differential equation
\begin{equation}\label{1.10}
\Delta Q + |Q|^{\frac{4}{d}} Q = Q.
\end{equation}
Such a solution is known to exist, see \cite{kwong1989uniqueness}. If $Q$ solves $(\ref{1.10})$, then $e^{it} Q(x)$ gives a global solution to $(\ref{1.1})$ when $\mu = -1$,
\begin{equation}\label{1.10.1}
i u_{t} + \Delta u = -|u|^{\frac{4}{d}} u, \qquad u(0,x) = u_{0}, \qquad u : I \times \mathbb{R}^{d} \rightarrow \mathbb{C},
\end{equation}
which does not scatter in either time direction. Furthermore, if $u(t,x)$ is a solution to $(\ref{1.10.1})$, then applying the pseudoconformal transformation to $u$,
\begin{equation}\label{1.11}
v(t,x) = \frac{1}{|t|^{d/2}} \bar{u}(\frac{1}{t}, \frac{x}{t}) e^{i \frac{|x|^{2}}{4t}},
\end{equation}
is also a solution to $(\ref{1.10.1})$. Applying the pseudoconformal transformation to $e^{it} Q(x)$ gives a solution to $(\ref{1.10.1})$ that blows up in finite time.

Furthermore, the mass $\| Q \|_{L^{2}}$ represents a blowup threshold. In the case when $\| u_{0} \|_{L^{2}} < \| Q \|_{L^{2}}$ and $u_{0} \in H^{1}$, \cite{weinstein1983nonlinear} proved that $(\ref{1.10.1})$ has a global solution using conservation of mass, energy, and the Gagliardo--Nirenberg inequality,
\begin{equation}\label{1.9}
\| f \|_{L^{2 + \frac{4}{d}}(\mathbb{R}^{d})}^{2 + \frac{4}{d}} \leq \frac{d + 2}{d} (\frac{\| f \|_{L^{2}(\mathbb{R}^{d})}}{\| Q \|_{L^{2}(\mathbb{R}^{d})}})^{\frac{4}{d}} \| \nabla f \|_{L^{2}(\mathbb{R}^{d})}^{2}.
\end{equation}
Plugging $(\ref{1.9})$ into $(\ref{1.7})$ when $\mu = -1$,
\begin{equation}
E(u(t)) \geq \frac{1}{2} \| \nabla u(t) \|_{L^{2}}^{2} (1 - \frac{\| u_{0} \|_{L^{2}}^{4/d}}{\| Q \|_{L^{2}}^{4/d}}).
\end{equation}
For initial data $u_{0} \in L^{2}$ satisfying $\| u_{0} \|_{L^{2}} < \| Q \|_{L^{2}}$, where $u_{0}$ need not lie in $H^{1}$, \cite{dodson2015global} proved global well-posedness and scattering.

Less is known about the focusing problem when $\| u_{0} \|_{L^{2}} = \| Q \|_{L^{2}}$. It is conjectured that $u(t,x) = e^{it} Q(x)$ and its pseudoconformal transformation are the only non-scattering solutions to $(\ref{1.10.1})$ when $\| u_{0} \|_{L^{2}} = \| Q \|_{L^{2}}$, modulo symmetries of $(\ref{1.1})$. The symmetries of $(\ref{1.1})$ include the scaling symmetry, which has already been discussed $(\ref{1.2})$, translation in space and time,
\begin{equation}\label{1.13}
u(t - t_{0}, x - x_{0}), \qquad t_{0} \in \mathbb{R}, \qquad x_{0} \in \mathbb{R}^{d},
\end{equation}
phase transformation,
\begin{equation}\label{1.14}
\forall \theta_{0} \in \mathbb{R}, \qquad e^{i \theta_{0}} u(t, x),
\end{equation}
and the Galilean transformation,
\begin{equation}\label{1.15}
e^{i \frac{\xi_{0}}{2} \cdot (x - \frac{\xi_{0}}{2} t)} u(t, x - \xi_{0} t), \qquad \xi_{0} \in \mathbb{R}^{d}.
\end{equation}

This conjecture was answered in the affirmative in all dimensions for finite time blowup solutions with finite energy initial data. See \cite{merle1992uniqueness} and \cite{merle1993determination}. This conjecture was also answered in the affirmative for a radially symmetric solution to $(\ref{1.10.1})$ in dimensions $d \geq 4$ that blow up in both time directions, but not necessarily in finite time. See \cite{killip2009characterization}.
\begin{remark}
Throughout this paper, blowup refers to failure to scatter, and could mean either finite or in infinite time, unless specified otherwise. From \cite{cazenave1989some}, \cite{cazenave1990cauchy}, failure to scatter forward in time is equivalent to
\begin{equation}
\| u \|_{L_{t,x}^{\frac{2(d + 2)}{d}}([0, \sup(I)) \times \mathbb{R}^{d})} = \infty,
\end{equation}
where $I$ is the maximal interval of existence of $u$.
\end{remark}

\begin{remark}
The pseudoconformal transformation of the solution $e^{it} Q(x)$ is a solution that blows up in one time direction but scatters in the other. By time reversal symmetry, it is possible to assume without loss of generality that the solution blows up forward in time. So \cite{merle1992uniqueness} and \cite{merle1993determination} proved that a finite energy, finite time blowup solution to $(\ref{1.10.1})$ must be a pseudoconformal transformation of $e^{it} Q(x)$. Meanwhile, \cite{killip2009characterization} showed that the only radial solution to $(\ref{1.10.1})$ that blows up in both time directions in dimensions $d \geq 4$ is the soliton $e^{it} Q$.
\end{remark}

More recently, \cite{fan20182} proved a sequential convergence result for radially symmetric solutions that may only blow up in one time direction.

\begin{theorem}\label{t1.1}
Assume that $u$ is a radial solution to the focusing, mass-critical nonlinear Schr{\"o}dinger equation, $(\ref{1.10.1})$, with $\| u_{0} \|_{L^{2}} = \| Q \|_{L^{2}}$, which does not scatter forward in time. Let $(T^{-}(u), T^{+}(u))$ be its lifespan, $T^{-}(u)$ could be $-\infty$ and $T^{+}(u)$ could be $+\infty$. Then there exists a sequence $t_{n} \nearrow T^{+}(u)$ and a family of parameters $\lambda_{\ast, n}$, $\gamma_{\ast, n}$ such that
\begin{equation}\label{1.16}
\lambda_{\ast, n}^{d/2} u(t_{n}, \lambda_{\ast, n} x) e^{-i \gamma_{\ast, n}} \rightarrow Q, \qquad \text{in} \qquad L^{2}.
\end{equation}
\end{theorem}
In fact, \cite{fan20182} proved Theorem $\ref{t1.1}$ for a larger class of initial data, data which is symmetric across $d$ linearly independent hyperplanes. In one dimension, there is no difference between radial initial data and symmetric initial data, but there is in higher dimensions.

In a previous paper, \cite{dodson20202}, we removed the symmetry assumption in dimension one. Here, we continue this study and remove the symmetry assumption in dimensions $d \geq 2$. In doing so, we must allow for translation, $(\ref{1.13})$, and Galilean symmetries, $(\ref{1.15})$, not just scaling and phase transformation symmetries.
\begin{theorem}\label{t1.2}
Assume $u$ is solution to $(\ref{1.1})$ with $\| u_{0} \|_{L^{2}} = \| Q \|_{L^{2}}$ which does not scatter forward in time. Let $(T^{-}(u), T^{+}(u))$ be its lifespan, $T^{-}(u)$ could be $-\infty$ and $T^{+}(u)$ could be $+\infty$. Then there exists a sequence $t_{n} \nearrow T^{+}(u)$ and a family of parameters $\lambda_{\ast, n}$, $\gamma_{\ast, n}$, $\xi_{\ast, n}$, $x_{\ast, n}$ such that
\begin{equation}\label{1.17}
\lambda_{\ast, n}^{d/2} e^{ix \xi_{\ast, n}} u(t_{n}, \lambda_{\ast, n} x + x_{\ast, n}) e^{-i \gamma_{\ast, n}} \rightarrow Q, \qquad \text{in} \qquad L^{2}.
\end{equation}
\end{theorem}

When $\| u_{0} \|_{L^{2}} > \| Q \|_{L^{2}}$, one can easily construct solutions to $(\ref{1.1})$ that blow up in finite time. Indeed, using the virial identity for a solution to $(\ref{1.1})$,
\begin{equation}\label{1.15.0}
\frac{d^{2}}{dt^{2}} \int |x|^{2} |u(t,x)|^{2} dx = 16 E(u_{0}),
\end{equation}
for $u_{0} \in H^{1}$, $\| |x| u_{0} \|_{L^{2}} < \infty$, $E(u_{0}) < 0$, $(\ref{1.15.0})$ implies that the variance $\int |x|^{2} |u(t,x)|^{2} dx$ is a concave function in time. Therefore, the variance can only be positive on some finite interval $(-T_{1}, T_{2})$, where $T_{1}$, $T_{2} < \infty$, which implies that the solution to $(\ref{1.1})$ with such initial data cannot exist outside the time interval $(-T_{1}, T_{2})$. Initial data $u_{0} = (1 + \epsilon) Q$ satisfies the above conditions for any $\epsilon > 0$.

For initial data with nonpositive energy and mass slightly above the ground state
\begin{equation}\label{1.18}
\| Q \|_{L^{2}} < \| u_{0} \|_{L^{2}} \leq \| Q \|_{L^{2}} + \alpha, \qquad \text{for some} \qquad \alpha > 0 \qquad \text{small},
\end{equation}
\cite{merle2006sharp} proved that after acting on the solution with the appropriate symmetries, $u(t,x)$ converges weakly to $Q$ as $t$ converges to the blowup time. Such solutions would include the above mentioned solutions with finite variance and negative energy that satisfy $(\ref{1.18})$.

This fact also holds for any solution to $(\ref{1.1})$ that satisfies $(\ref{1.18})$ and fails to scatter. Once again, we generalize a result of \cite{fan20182} to the non-symmetric case in dimensions $d \geq 2$.
\begin{theorem}\label{t1.3}
Assume $u$ is a solution to $(\ref{1.1})$ with $u_{0}$ satisfying $(\ref{1.18})$, which does not scatter forward in time. Let $(T^{-}(u), T^{+}(u))$ be the lifespan of the solution. Then there exists a sequence of times $t_{n} \nearrow T^{+}(u)$ and a family of parameters $\lambda_{\ast, n}$, $\gamma_{\ast, n}$, $\xi_{\ast, n}$, $x_{\ast, n}$ such that
\begin{equation}\label{1.19}
\lambda_{\ast, n}^{d/2} e^{ix \cdot \xi_{\ast, n}} u(t_{n}, \lambda_{\ast, n} x + x_{\ast, n}) e^{-i \gamma_{\ast, n}} \rightharpoonup Q, \qquad \text{weakly in} \qquad L^{2}.
\end{equation}
\end{theorem}

\section{A Preliminary reduction}
The scattering result of \cite{dodson2015global} implies that a non-scattering solution to $(\ref{1.1})$ with $\| u_{0} \|_{L^{2}} = \| Q \|_{L^{2}}$ is a minimal mass blowup solution to $(\ref{1.1})$. Therefore, it is possible to make a reduction to a to an almost periodic solution in proving Theorem $\ref{t1.2}$. Let $t_{n} \nearrow T^{+}(u)$ be a sequence of times. Making a profile decomposition, after passing to a subsequence, for all $J$,
\begin{equation}\label{2.1}
u(t_{n}) = \sum_{j = 1}^{J} g_{n}^{j} [e^{i t_{n}^{j} \Delta} \phi^{j}] + w_{n}^{J},
\end{equation}
where $g_{n}^{j}$ is the group action
\begin{equation}\label{2.1.1}
g_{n}^{j} \phi^{j} = \lambda_{n, j}^{d/2} e^{ix \cdot \xi_{n,j}} e^{i \gamma_{n,j}} \phi^{j}(\lambda_{n,j} x + x_{n,j}),
\end{equation}
and
\begin{equation}
\lim_{J \rightarrow \infty} \limsup_{n \rightarrow \infty} \| e^{it \Delta} w_{n}^{J} \|_{L_{t,x}^{\frac{2(d + 2)}{d}}(\mathbb{R} \times \mathbb{R}^{d})} = 0.
\end{equation}
Since $u$ is a minimal mass blowup solution, $\phi^{j} = 0$ for $j \geq 2$, $\| \phi^{1} \|_{L^{2}} = \| Q \|_{L^{2}}$, and $\| w_{n}^{J} \|_{L^{2}} \rightarrow 0$ as $n \rightarrow \infty$. See \cite{dodson2019defocusing}, \cite{killip2013nonlinear}, or \cite{tao2008minimal} for a detailed treatment of the profile decomposition for minimal mass blowup solutions. Thus, it will be convenient to drop the $j$ notation and simply write,
\begin{equation}\label{2.1.2}
u(t_{n}) = g_{n} \phi + w_{n}.
\end{equation}
\begin{remark}
The disappearance of $e^{it_{n}^{1} \Delta}$ in $(\ref{2.1.2})$ will be explained soon.
\end{remark}

Now let $v$ be the solution to $(\ref{1.1})$ with initial data $\phi$, and let $I$ be the maximal interval of existence of $v$. Since
\begin{equation}\label{2.1.3}
\lim_{n \rightarrow \infty} \| u \|_{L_{t,x}^{\frac{2(d + 2)}{d}}((T^{-}(u), t_{n}) \times \mathbb{R}^{d})} = \infty, \qquad \text{and} \qquad \| u \|_{L_{t,x}^{\frac{2(d + 2)}{d}}((t_{n}, T^{+}(u)) \times \mathbb{R}^{d})} = \infty \qquad \forall n,
\end{equation}
\begin{equation}\label{2.2}
\| v \|_{L_{t,x}^{\frac{2(d + 2)}{d}}([0, \sup(I)) \times \mathbb{R}^{d})} = \| v \|_{L_{t,x}^{\frac{2(d + 2)}{d}}((\inf(I), 0] \times \mathbb{R}^{d})} = \infty.
\end{equation}
\begin{remark}
Equation $(\ref{2.1.3})$ is also the reason that it is unnecessary to allow for the possibility of terms like $[e^{it_{n}^{j} \Delta} \phi^{j}]$ in $(\ref{2.1})$ in place of $\phi^{j}$, where $t_{n}^{j} \rightarrow \pm \infty$. If $t_{n}^{j}$ converges along a subsequence to some $t_{0}^{j} \in \mathbb{R}$, then $\phi^{j}$ can be replaced by $e^{i t_{0}^{j} \Delta} \phi^{j}$.
\end{remark}

\begin{theorem}\label{t2.1}
To prove Theorem $\ref{t1.2}$, it suffices to prove that there exists a sequence $s_{m} \nearrow \sup(I)$, $s_{m} \geq 0$, such that
\begin{equation}\label{2.2.1}
g(s_{m}) v(s_{m}) \rightarrow Q, \qquad \text{in} \qquad L^{2}.
\end{equation}
\end{theorem}
\begin{proof}
Suppose $g(s_{m}) v(s_{m}) \rightarrow Q$ in $L^{2}$. For any $m$ let $s_{m} \in I$ be such that
\begin{equation}\label{2.4}
\| g(s_{m}) v(s_{m}) - Q \|_{L^{2}} \leq 2^{-m}.
\end{equation}
Next, observe that $(\ref{2.1})$ implies
\begin{equation}\label{2.3}
e^{i \xi_{n} \cdot x} e^{i \gamma_{n}} \lambda_{n}^{d/2} u(t_{n}, \lambda_{n} x + x_{n}) \rightarrow \phi, \qquad \text{in} \qquad L^{2},
\end{equation}
and by $(\ref{1.15})$ and perturbation theory, for a fixed $m$, for $n$ sufficiently large,
\begin{equation}\label{2.5}
\aligned
\|  e^{-i \xi_{n}^{2} s_{m}} e^{i \xi_{n} \cdot x} e^{i \gamma_{n}} \lambda_{n}^{d/2} u(t_{n} + \lambda_{n}^{2} s_{m}, \lambda_{n} x + x_{n} - 2 \xi_{n} \lambda_{n} s_{m}) - v(s_{m}) \|_{L^{2}} \\ \leq C(s_{m}) \| e^{i \xi_{n} \cdot x} e^{i \gamma_{n}} \lambda_{n}^{d/2} u(t_{n}, \lambda_{n} x + x_{n}) - \phi \|_{L^{2}}.
\endaligned
\end{equation}
Therefore, by $(\ref{2.4})$, $(\ref{2.5})$, and the triangle inequality,
\begin{equation}\label{2.6}
\aligned
\| g(s_{m})(\lambda_{n}^{d/2} e^{-i \xi_{n}^{2} s_{m}} e^{i \xi_{n} \cdot x} e^{i \gamma_{n}} u(t_{n} + \lambda_{n}^{2} s_{m}, \lambda_{n} x + x_{n} - 2 \xi_{n} \lambda_{n} s_{m})) - Q \|_{L^{2}} \\ \leq C(s_{m}) \| e^{i \xi_{n} \cdot x} e^{i \gamma_{n}} \lambda_{n}^{d/2} u(t_{n}, \lambda_{n} x + x_{n}) - \phi \|_{L^{2}} + 2^{-m}.
\endaligned
\end{equation}
Since $g(s_{m})$ is also of the form $(\ref{2.1.1})$, there exists a group action $g_{n,m}$ of the form $(\ref{2.1.1}$ such that
\begin{equation}\label{2.6.1}
g(s_{m})(\lambda_{n}^{d/2} e^{-i \xi_{n}^{2} s_{m}} e^{i \xi_{n} \cdot x} e^{i \gamma_{n}} u(t_{n} + \lambda_{n}^{2} s_{m}, \lambda_{n} x + x_{n} - 2 \xi_{n} \lambda_{n} s_{m})) = g_{n,m} u(t_{n} + \lambda_{n}^{2} s_{m}, x).
\end{equation}
Equation $(\ref{2.6})$ implies
\begin{equation}\label{2.7}
\lim_{m, n \rightarrow \infty} \| g_{n,m} u(t_{n} + \lambda_{n}^{2} s_{m}, x)  - Q \|_{L^{2}} = 0.
\end{equation}
Since $t_{n} \nearrow T^{+}(u)$ and $s_{m} \geq 0$, $t_{n} + \lambda_{n}^{2} s_{m} \nearrow T^{+}(u)$, which implies Theorem $\ref{t1.2}$, assuming that $(\ref{2.2.1})$ is true.
\end{proof}

Now then, since $v(s)$ blows up in both time directions, $(\ref{2.2})$ holds, and $\| v \|_{L^{2}} = \| Q \|_{L^{2}}$, we can use the result of \cite{tao2008minimal} to prove that $v$ is almost periodic. That is, for all $s \in I$, there exist $\lambda(s) > 0$, $\xi(s) \in \mathbb{R}^{d}$, $x(s) \in \mathbb{R}^{d}$, and $\gamma(s) \in \mathbb{R}$ such that
\begin{equation}\label{2.8}
\lambda(s)^{-d/2} e^{ix \cdot \xi(s)} e^{i \gamma(s)} v(s, \frac{x - x(s)}{\lambda(s)}) \in K,
\end{equation}
where $K$ is a fixed precompact subset of $L^{2}$. Therefore, in the case when $\| u_{0} \|_{L^{2}} = \| Q \|_{L^{2}}$, it only remains to prove sequential convergence to $Q$ for this solution $v$.
\begin{theorem}\label{t2.2}
There exists a sequence $s_{m} \nearrow \sup(I)$ and a sequence of group actions $g(s_{m})$ of the form $(\ref{2.1.1})$ such that
\begin{equation}\label{2.9}
\| g(s_{m}) v(s_{m}) - Q \|_{L^{2}} \rightarrow 0.
\end{equation}
\end{theorem}
The proof of this fact will occupy the next two sections. Since $v$ is almost periodic, the tools used in \cite{dodson2015global} are available in this case as well.
\begin{remark}
In order for notation to align with notation in prior works, such as \cite{dodson2015global}, it will be convenient to relabel so that $v$ is now denoted $u$, and $s$ now denoted $t$.
\end{remark}

Similarly, when proving Theorem $\ref{t1.3}$, we use Lemma $4.2$ from \cite{fan20182} to reduce to an almost periodic solution.
\begin{lemma}\label{l2.3}
Let $u$ be a solution to $(\ref{1.1})$ satisfying the assumptions of Theorem $\ref{t1.3}$. Then there exists a sequence $t_{n} \nearrow T^{+}(u)$ such that $u(t_{n})$ admits the profile decomposition in $(\ref{2.1})$, and there is a unique profile $\phi_{1}$, such that $\| \phi_{1} \|_{L^{2}} \geq \| Q \|_{L^{2}}$ and the solution $v$ to $(\ref{1.1})$ with initial data $\phi_{1}$ is an almost periodic solution to $(\ref{1.1})$ that does not scatter forward or backward in time.
\end{lemma}

In this case as well, it suffices to show that passing to a subsequence, $g(s_{m}) v(s_{m}) \rightharpoonup Q$, using similar arguments as in the case when $\| u \|_{L^{2}} = \| Q \|_{L^{2}}$. Indeed, by asymptotic orthogonality of the profile decomposition $(\ref{2.1})$, for $n(m)$ sufficiently large,
\begin{equation}\label{2.10}
(g_{n}^{1})^{-1} g(s_{m}) u(t_{n} + \lambda_{n}^{2} s_{m}, x) = g(s_{m}) v(s_{m}) + R_{n(m), m} + g(s_{m}) \sum_{j = 2}^{J} g_{n(m)}^{j} \Phi^{j} + g(s_{m}) w_{n}^{J},
\end{equation}
where $\| R_{n(m), m} \|_{L^{2}} \rightarrow 0$ as $m \rightarrow \infty$, $g_{n(m)}^{j} \Phi^{j}$ are the solutions to $(\ref{1.1})$ with data $g_{n(m)}^{j} \phi^{j}$ or that scatter forward or backward in time to $\phi^{j}$, and $w_{n}^{J}$ is the solution to $(\ref{1.1})$ with initial data $w_{n}^{J}$. Furthermore, asymptotic orthogonality implies that
\begin{equation}\label{2.11}
g(s_{m}) \sum_{j = 2}^{J} g_{n(m)}^{j} \Phi^{j} + g(s_{m}) w_{n}^{J} \rightharpoonup 0, \qquad \text{in} \qquad L^{2},
\end{equation}
which proves the reduction.

\begin{remark}
Since $(\ref{2.1})$ is not a profile decomposition for a minimal mass blowup solution, it is possible that $t_{n}^{j} \rightarrow \pm \infty$ for $j \geq 2$.
\end{remark}

\section{Proof of Theorem $\ref{t2.2}$ when $\lambda(t) = 1$ and $d = 2$}
It will be convenient to begin by discussing the $\lambda(t) = 1$ case in dimension $d = 2$, before generalizing the argument to higher dimensions and variable $\lambda(t)$. When $\lambda(t) = 1$, the solution $u$ is global in both time directions, $I = \mathbb{R}$. Following \cite{dodson2015global} and \cite{dodson20202}, we will use the interaction Morawetz estimate
\begin{equation}\label{3.1}
M(t) = \int \int |Iu(t, y)|^{2} Im[\bar{Iu} \nabla Iu](t,x) \cdot (x - y) \psi(x - y) dx dy,
\end{equation}
where $I$ is the Fourier truncation operator $P_{\leq T}$, $T = 2^{k}$ for some $k \in \mathbb{Z}_{\geq 0}$. As in \cite{dodson2015global}, $\psi(|x - y|)$ is a radial function,
\begin{equation}\label{3.2}
\psi(x) = \frac{1}{|x - y|} \int_{0}^{|x - y|} \phi(s) ds,
\end{equation}
where $\phi(|x|)$ is a radial function given by
\begin{equation}\label{3.3}
\phi(|x - y|) = \frac{1}{R^{2}} \int \chi^{2}(\frac{x - y - s}{R}) \chi^{2}(\frac{s}{R}) ds = \frac{1}{R^{2}} \int \chi^{2}(\frac{x - s}{R}) \chi^{2}(\frac{s - y}{R}) ds = \frac{1}{R^{2}} \int \chi^{2}(\frac{x - s}{R}) \chi^{2}(\frac{y - s}{R}) ds,
\end{equation}
where $\chi$ is a radial, smooth, compactly supported function, $\chi(x) = 1$ for $|x| \leq 1$ and $\chi(x)$ is supported on $|x| \leq 2$. In addition, $\chi(|x|)$ is decreasing as a function of the radius. $R$ is a large, fixed constant that will be allowed to go to infinity as $T \rightarrow \infty$.
\begin{remark}
$\phi$ decreasing as a function of the radius implies that $\psi$ is decreasing as a function of the radius.
\end{remark}

By direct computation,
\begin{equation}\label{3.4}
\aligned
\frac{d}{dt} M(t) = 2 \int \int |Iu(t,y)|^{2} Re[\partial_{j} \bar{Iu} \partial_{k} Iu](t,x) [\delta_{jk} \psi(x - y) + \frac{(x - y)_{j} (x - y)_{k}}{|x - y|} \psi'(x - y)] dx dy \\ -2 \int \int Im[\bar{Iu} \partial_{k} Iu](t,y) Im[\bar{Iu} \partial_{j} Iu](t,x) [\delta_{jk} \psi(x - y) + \frac{(x - y)_{j} (x - y)_{k}}{|x - y|} \psi'(x - y)] dx dy \\
+ \frac{1}{2} \int \int |Iu(t,y)|^{2} |Iu(t,y)|^{2} [\Delta \phi(x - y) + \Delta \psi(x - y)] dx dy \\
- \int \int |Iu(t,y)|^{2} |Iu(t,x)|^{4} [\psi(x - y) + \frac{1}{2} \psi'(x - y) |x - y|] dx dy + \mathcal E,
\endaligned
\end{equation}
where $\mathcal E$ are the error terms arising from $\mathcal N$,
\begin{equation}\label{3.4.1}
i Iu_{t} + \Delta I u + F(Iu) = F(Iu) - I F(u) = \mathcal N.
\end{equation}
It is known from \cite{dodson2016global2} that
\begin{equation}\label{3.4.2}
\int_{0}^{T} \mathcal N dt \lesssim R o(T),
\end{equation}
and
\begin{equation}\label{3.4.3}
\sup_{t \in [0, T]} |M(t)| \lesssim R o(T).
\end{equation}
Therefore, choosing $R \nearrow \infty$ sufficiently slowly,
\begin{equation}\label{3.4.4}
\lim_{T \rightarrow \infty} \frac{R o(T)}{T} = 0.
\end{equation}

By direct computation,
\begin{equation}\label{3.5}
\phi(x) = \frac{1}{R^{2}} \int \chi^{2}(\frac{x - s}{R}) \chi^{2}(\frac{s}{R}) ds \sim 1,
\end{equation}
for $|x| \leq R$, $\phi(x)$ is supported on the set $|x| \leq 4R$, and $\phi(x)$ is a radially symmetric function that is decreasing as $|x| \rightarrow \infty$. Therefore, $(\ref{3.2})$ implies that
\begin{equation}\label{3.6}
|\psi(x)| \lesssim \frac{R}{|x|}, \qquad \text{for all} \qquad x \in \mathbb{R}^{2}.
\end{equation}
Also, by direct computation,
\begin{equation}\label{3.7}
\Delta \phi(x) = \frac{1}{R^{2}} \int \Delta \chi^{2}(\frac{x - s}{R}) \chi^{2}(\frac{s}{R}) ds \lesssim \frac{1}{R^{2}}.
\end{equation}
Next, by the same calculations that give $(\ref{3.6})$,
\begin{equation}\label{3.7.1}
\Delta \psi(x) \lesssim \frac{R}{|x|^{3}},
\end{equation}
so $|\Delta \psi(x)| \lesssim \frac{1}{R^{2}}$ for $|x| \gtrsim R$. By the fundamental theorem of calculus, since $\phi'(0) = 0$, by $(\ref{3.2})$,
\begin{equation}\label{3.7.2}
\psi(r) = \phi(0) + \frac{1}{r} \int_{0}^{r} \int_{0}^{s} (s - t) \phi''(t) dt ds,
\end{equation}
so by $(\ref{3.7})$, $|\Delta \psi(x)| \lesssim \frac{1}{R^{2}}$ for $|x| \lesssim R$.
Therefore,
\begin{equation}\label{3.8}
 \frac{1}{2} \int \int |Iu(t,y)|^{2} |Iu(t,y)|^{2} [\Delta \phi(x - y) + \Delta \psi(x - y)] dx dy \lesssim \frac{1}{R^{2}} \| u \|_{L^{2}}^{4}.
 \end{equation}

Next, decompose
\begin{equation}\label{3.8.1}
\delta_{jk} \psi(x - y) + \frac{(x - y)_{j} (x - y)_{k}}{|x - y|} \psi'(x - y) = \delta_{jk} \phi(x - y) + \delta_{jk} [\psi(x - y) - \phi(x - y)] + \frac{(x - y)_{j} (x - y)_{k}}{|x - y|} \psi'(x - y).
\end{equation}
By $(\ref{3.2})$,
\begin{equation}\label{3.8.2}
(\ref{3.8.1}) = \delta_{jk} \phi(x - y) - \delta_{jk} |x - y| \psi'(|x - y|) + \frac{(x - y)_{j} (x - y)_{k}}{|x - y|} \psi'(x - y).
\end{equation}
Now then,
\begin{equation}\label{3.9}
\aligned
-\int \int Im[\bar{Iu} \partial_{k} Iu] Im[\bar{Iu} \partial_{j} Iu] \delta_{jk} \phi(x - y) dx dy + \int \int |Iu(t,y)|^{2} |\nabla Iu(t,x)|^{2} \phi(x - y) dx dy \\
= \frac{1}{R^{2}} \int (\int \chi^{2}(\frac{y - s}{R}) Im[\bar{Iu} \partial_{j} Iu] dy)(\int \chi^{2}(\frac{x - s}{R}) Im[\bar{Iu} \partial_{j} Iu] dx) ds \\ + \frac{1}{R^{2}} \int (\int \chi^{2}(\frac{y - s}{R}) |Iu(t,y)|^{2} dy)(\int \chi^{2}(\frac{x - s}{R}) |\nabla Iu(t,x)|^{2} dx) ds.
\endaligned
\end{equation}
Fix $s \in \mathbb{R}^{2}$. For any $\xi \in \mathbb{R}^{2}$ and $j \in \mathbb{Z}$,
\begin{equation}\label{3.10}
\int \chi^{2}(\frac{y - s}{R}) Im[\overline{e^{iy \xi} Iu} \partial_{j}(e^{iy \xi} Iu)] dy = \int \chi^{2}(\frac{y - s}{R}) Im[\bar{Iu} \partial_{j} Iu] dy + \xi_{j} \int \chi^{2}(\frac{y - s}{R}) |Iu(t,y)|^{2} dy,
\end{equation}
and
\begin{equation}\label{3.11}
\aligned
\int \chi^{2}(\frac{x - s}{R}) |\partial_{j} (e^{ix \xi} Iu)|^{2} dx = \xi_{j}^{2} \int \chi^{2}(\frac{x - s}{R}) |Iu|^{2} dx \\ + 2 \xi_{j} \int \chi^{2}(\frac{x - s}{R}) Im[\bar{Iu} \partial_{j} Iu] dx + \int \chi^{2}(\frac{x - s}{R}) |\partial_{j} Iu|^{2} dx.
\endaligned
\end{equation}
Therefore, $(\ref{3.9})$ is invariant under the Galilean transformation, so it is convenient to choose $\xi(s)$ such that $(\ref{3.10}) = 0$. For notational convenience, let 
\begin{equation}\label{3.12}
v_{s} = e^{ix \xi(s)} Iu.
\end{equation}
Then by the fundamental theorem of calculus and $(\ref{3.4.2})$--$(\ref{3.12})$, if $R \nearrow \infty$ as $T \nearrow \infty$,
\begin{equation}\label{3.13}
\aligned
2 \int_{0}^{T} \frac{1}{R^{2}} \int (\int \chi^{2}(\frac{y - s}{R}) |v_{s}(t,y)|^{2}) (\int \chi^{2}(\frac{x - s}{R}) |\nabla (v_{s})(t,x)|^{2} dx) ds dt \\
2 \int \int |Iu(t,y)|^{2} Re[\partial_{j} \bar{Iu} \partial_{k} Iu](t,x) [\delta_{jk} |x - y| \psi'(x - y) + \frac{(x - y)_{j} (x - y)_{k}}{|x - y|} \psi'(x - y)] dx dy \\ -2 \int \int Im[\bar{Iu} \partial_{k} Iu](t,y) Im[\bar{Iu} \partial_{j} Iu](t,x) [\delta_{jk} |x - y| \psi'(x - y) + \frac{(x - y)_{j} (x - y)_{k}}{|x - y|} \psi'(x - y)] dx dy \\
- \int_{0}^{T} \frac{1}{R^{2}} \int (\int \chi^{2}(\frac{y - s}{R}) |v_{s}(t,y)|^{2}) (\int \chi^{2}(\frac{x - s}{R}) |v_{s}(t,x)|^{4} dx) ds dt \\
- \frac{1}{2} \int_{0}^{T} \int |Iu(t,y)|^{2} [\psi(x - y) - \phi(x - y)] |Iu(t,x)|^{4} dx dy dt \lesssim R o(T).
\endaligned
\end{equation}
Following the computations in \cite{dodson2015global} for the angular derivatives in dimensions $d \geq 2$,
\begin{equation}\label{3.13.0}
\aligned
2 \int \int |Iu(t,y)|^{2} Re[\partial_{j} \bar{Iu} \partial_{k} Iu](t,x) [\delta_{jk} |x - y| \psi'(x - y) + \frac{(x - y)_{j} (x - y)_{k}}{|x - y|} \psi'(x - y)] dx dy \\ -2 \int \int Im[\bar{Iu} \partial_{k} Iu](t,y) Im[\bar{Iu} \partial_{j} Iu](t,x) [\delta_{jk} |x - y| \psi'(x - y) + \frac{(x - y)_{j} (x - y)_{k}}{|x - y|} \psi'(x - y)] dx dy \geq 0.
\endaligned
\end{equation}
Therefore, by $(\ref{3.13})$,
\begin{equation}\label{3.13.0.1}
\aligned
2 \int_{0}^{T} \frac{1}{R^{2}} \int (\int \chi^{2}(\frac{y - s}{R}) |v_{s}(t,y)|^{2}) (\int \chi^{2}(\frac{x - s}{R}) |\nabla (v_{s})(t,x)|^{2} dx) ds dt \\
- \int_{0}^{T} \frac{1}{R^{2}} \int (\int \chi^{2}(\frac{y - s}{R}) |v_{s}(t,y)|^{2}) (\int \chi^{2}(\frac{x - s}{R}) |v_{s}(t,x)|^{4} dx) ds dt \\
- \frac{1}{2} \int_{0}^{T} \int |Iu(t,y)|^{2} [\psi(x - y) - \phi(x - y)] |Iu(t,x)|^{4} dx dy dt \lesssim R o(T).
\endaligned
\end{equation}

By the Arzela--Ascoli theorem and $(\ref{2.8})$, for any $\eta > 0$, there exists $C(\eta) < \infty$ such that
\begin{equation}\label{3.13.1}
\int_{|x - x(t)| \geq \frac{C(\eta)}{\lambda(t)}} |u(t,x)|^{2} dx < \eta^{2}.
\end{equation}
By H{\"o}lder's inequality, Strichartz estimates, and $\lambda(t) = 1$,
\begin{equation}\label{3.13.2}
\aligned
\int_{a}^{a + 1} \int_{|y - y(t)| \geq C(\eta)} |Iu(t,y)|^{2} \int |Iu(t,x)|^{4} dx dy dt + \int_{a}^{a + 1} \int |Iu(t,y)|^{2} \int_{|x - x(t)| \geq C(\eta)} |Iu(t,x)|^{4} dx dy dt \\ \lesssim \eta^{2} \| u \|_{L_{t,x}^{4}([a, a + 1] \times \mathbb{R}^{2})}^{4} + \eta \| u \|_{L_{t}^{3} L_{x}^{6}([a, a + 1] \times \mathbb{R}^{2})}^{3} \lesssim \eta.
\endaligned
\end{equation}
Finally, by $(\ref{3.2})$ and the fundamental theorem of calculus,
\begin{equation}\label{3.13.3}
\int_{a}^{a + 1} \int_{|y - y(t)| \leq C(\eta)} \int_{|x - x(t)| \leq C(\eta)} |Iu(t,y)|^{2} [\psi(x - y) - \phi(x - y)] |Iu(t,x)|^{4} dx dy dt \lesssim \frac{C(\eta)}{R} \| u \|_{L_{t,x}^{4}([a, a + 1] \times \mathbb{R}^{2})}^{4}.
\end{equation}
Therefore,
\begin{equation}\label{3.13.4}
\frac{1}{2} \int_{0}^{T} \int |Iu(t,y)|^{2} [\psi(x - y) - \phi(x - y)] |Iu(t,x)|^{4} dx dy dt \lesssim \eta T + \frac{C(\eta)}{R} T,
\end{equation}
so
\begin{equation}\label{3.13.5}
\aligned
2 \int_{0}^{T} \frac{1}{R^{2}} \int (\int \chi^{2}(\frac{y - s}{R}) |v_{s}(t,y)|^{2}) (\int \chi^{2}(\frac{x - s}{R}) |\nabla (v_{s})(t,x)|^{2} dx) ds dt \\
- \int_{0}^{T} \frac{1}{R^{2}} \int (\int \chi^{2}(\frac{y - s}{R}) |v_{s}(t,y)|^{2}) (\int \chi^{2}(\frac{x - s}{R}) |v_{s}(t,x)|^{4} dx) ds dt \lesssim R o(T) + \eta T + \frac{C(\eta)}{R} T.
\endaligned
\end{equation}

In \cite{dodson20202}, following  \cite{merle2001existence}, we used the fundamental theorem of calculus to obtain a bound on $|u|^{6}$ far away from the interval $|x - x(t)| \leq C(\eta)$. Here, we will use a computation from \cite{merle1993determination} in two dimensions. Let $f \in H^{1}(\mathbb{R}^{2})$ be any function and fix some $s \in \mathbb{R}^{2}$. By the fundamental theorem of calculus,
\begin{equation}\label{3.14}
\aligned
\int \chi(\frac{x - s}{R}) |f(x_{1}, x_{2})|^{4} dx_{1} dx_{2} \leq \int (\int |f(x_{1}, x_{2})|^{2} \chi(\frac{x - s}{R})) \cdot (\sup_{x_{1} \in \mathbb{R}} \chi(\frac{x - s}{R}) |f(x_{1}, x_{2})|^{2}) dx_{2} \\
\leq \int (\int |f(x_{1}, x_{2})|^{2} \chi(\frac{x - s}{R}) dx_{1})(\int |\partial_{x_{1}}(\chi(\frac{x - s}{R}) |f(x_{1}, x_{2})|^{2}|) dx_{1}) dx_{2} \\
\lesssim \int (\int |f(x_{1}, x_{2})|^{2} \chi(\frac{x - s}{R}) dx_{1}) \cdot (\int \chi(\frac{x - s}{R})^{2} |\partial_{x_{1}} f(x_{1}, x_{2})|^{2} dx_{1})^{1/2} (\int |f(x_{1}, x_{2})|^{2} dx_{1})^{1/2} dx_{2} \\
+ \int (\int |f(x_{1}, x_{2})|^{2} \chi(\frac{x - s}{R}) dx_{1}) \cdot (\frac{1}{R} \int |f(x_{1}, x_{2})|^{2} dx_{1}) dx_{2}.
\endaligned
\end{equation}

Again by the fundamental theorem of calculus and H{\"o}lder's inequality,
\begin{equation}\label{3.15}
\aligned
\int (\int |f(x_{1}, x_{2})|^{2} \chi(\frac{x - s}{R}) dx_{1}) \cdot (\int \chi(\frac{x - s}{R})^{2} |\partial_{x_{1}} f(x_{1}, x_{2})|^{2} dx_{1})^{1/2} (\int |f(x_{1}, x_{2})|^{2} dx_{1})^{1/2} dx_{2} \\
\lesssim \| \chi(\frac{x - s}{R}) \nabla f(x) \|_{L^{2}} \| f \|_{L^{2}} \cdot \sup_{x_{2}} (\int |f(x_{1}, x_{2})|^{2} \chi(\frac{x - s}{R}) dx_{1}) \\
\lesssim \| \chi(\frac{x - s}{R}) \nabla f(x) \|_{L^{2}} \| f \|_{L^{2}} \cdot (\int \int |\partial_{x_{2}} f(x_{1}, x_{2})| |f(x_{1}, x_{2})| \chi(\frac{x - s}{R}) dx_{1} dx_{2}) \\
+ \frac{1}{R} \| \chi(\frac{x - s}{R}) \nabla f(x) \|_{L^{2}} \| f \|_{L^{2}} \cdot (\int \int |f(x_{1}, x_{2})|^{2}) dx_{1} dx_{2}) \\ \lesssim \| \chi(\frac{x - s}{R}) \nabla f \|_{L^{2}}^{2} \| f \|_{L^{2}}^{2} + \frac{1}{R} \| \chi(\frac{x - s}{R}) \nabla f \|_{L^{2}} \| f \|_{L^{2}}^{3}.
\endaligned
\end{equation}
Also by H{\"o}lder's inequality and the fundamental theorem of calculus,
\begin{equation}\label{3.16}
\aligned
\int (\int |f(x_{1}, x_{2})|^{2} \chi(\frac{x - s}{R}) dx_{1}) \cdot (\frac{1}{R} \int |f(x_{1}, x_{2})|^{2} dx_{1}) dx_{2} \lesssim \frac{1}{R} \| f \|_{L^{2}}^{2} \cdot \sup_{x_{2}} (\int |f(x_{1}, x_{2})|^{2} \chi(\frac{x - s}{R}) dx_{1}) \\
\lesssim \frac{1}{R^{2}} \| f \|_{L^{2}}^{4} + \frac{1}{R} \| f \|_{L^{2}}^{3} \| \chi(\frac{x - s}{R}) \nabla f \|_{L^{2}}.
\endaligned
\end{equation}
For a fixed $t$, let $f = (1 - \chi(\frac{x - x(t)}{C(\eta)})) v_{s}(t,x)$. By $(\ref{3.13.1})$, $\| f \|_{L^{2}} \leq \eta$, and by the product rule,
\begin{equation}\label{3.16.1}
\| \chi(\frac{x - s}{R}) \nabla f \|_{L^{2}} \leq \| \chi(\frac{x - s}{R}) \nabla v_{s} \|_{L^{2}} + \frac{1}{C(\eta)} \| \chi'(\frac{x - x(t)}{C(\eta)}) v_{s} \|_{L^{2}} \leq \| \chi(\frac{x - s}{R}) \nabla v_{s} \|_{L^{2}} + \frac{\eta}{C(\eta)}.
\end{equation}

Therefore, since $\frac{1}{R^{2}} \int (\int \chi^{2}(\frac{y - s}{R}) |v_{s}(t,y)|^{2} dy) ds \lesssim \| u \|_{L^{2}}^{2}$,
\begin{equation}\label{3.17}
\aligned
2 \int_{0}^{T} \frac{1}{R} \int (\int \chi^{2}(\frac{y - s}{R}) |v_{s}(t,y)|^{2}) (\int \chi^{2}(\frac{x - s}{R}) |\nabla (v_{s})(t,x)|^{2} dx) ds dt \\
- \int_{0}^{T} \frac{1}{R} \int (\int \chi^{2}(\frac{y - s}{R}) |v_{s}(t,y)|^{2}) (\int_{|x - x(t)| \leq 2C(\eta)} \chi^{2}(\frac{x - s}{R}) |v_{s}(t,x)|^{4} dx) ds dt \\ \lesssim R o(T) + \eta T + \frac{C(\eta)}{R} T + \frac{\eta^{4}}{R^{2}} T + \frac{\eta^{2}}{R^{2}} \int_{0}^{T} \int (\int \chi^{2}(\frac{y - s}{R}) |v_{s}|^{2} dy) (\int \chi^{2}(\frac{x - s}{R}) |\nabla (v_{s})|^{2} dx) ds dt.
\endaligned
\end{equation}
When $\eta > 0$ is sufficiently small,
\begin{equation}\label{3.26}
 \frac{\eta^{2}}{R^{2}} \int_{0}^{T} \int (\int \chi^{2}(\frac{y - s}{R}) |v_{s}|^{2} dy) (\int \chi^{2}(\frac{x - s}{R}) |\nabla (v_{s})|^{2} dx) ds dt
 \end{equation}
 can be absorbed into the left hand side of $(\ref{3.17})$, proving that
 \begin{equation}\label{3.27}
\aligned
2 \int_{0}^{T} \frac{1}{R} \int (\int \chi^{2}(\frac{y - s}{R}) |v_{s}(t,y)|^{2}) (\int \chi^{2}(\frac{x - s}{R}) |\nabla (v_{s})(t,x)|^{2} dx) ds dt \\
- \int_{0}^{T} \frac{1}{R} \int (\int \chi^{2}(\frac{y - s}{R}) |v_{s}(t,y)|^{2}) (\int_{|x - x(t)| \leq 2C(\eta)} \chi^{2}(\frac{x - s}{R}) |v_{s}(t,x)|^{4} dx) ds dt \\ \lesssim R o(T) + \eta T + \frac{C(\eta)}{R} T+  \frac{\eta^{4}}{R^{2}} T \\ + \frac{ \eta^{2}}{R^{2}} \int_{0}^{T} \int (\int \chi^{2}(\frac{y - s}{R}) |v_{s}(t,y)|^{2} dy) (\int_{|x - x(t)| \leq 2C(\eta)} \chi^{2}(\frac{x - s}{R}) |v_{s}(t,x)|^{4} dx) ds dt.
\endaligned
\end{equation}

Now choose $|x_{\ast} - x(t)| \leq 4 C(\eta)$ such that
\begin{equation}\label{3.17.1}
\chi(\frac{x_{\ast} - s}{R}) = \inf_{|x - x(t)| \leq 4 C(\eta)} \chi(\frac{x - s}{R}).
\end{equation}
As in \cite{dodson20202}, the fundamental theorem of calculus implies that for $|x - x(t)| \leq 2 C(\eta)$,
\begin{equation}\label{3.18}
\chi^{2}(\frac{x - s}{R}) = \chi^{2}(\frac{x_{\ast} - s}{R}) + O(\frac{C(\eta)}{R}).
\end{equation}
Therefore,
\begin{equation}\label{3.20}
\aligned
\frac{1}{R^{2}} \int (\int \chi^{2}(\frac{y - s}{R}) |v_{s}(t,y)|^{2}) (\int_{|x - x(t)| \leq 2C(\eta)} \chi^{2}(\frac{x - s}{R}) |v_{s}(t,x)|^{4} dx) ds \\ \leq \frac{1}{R^{2}} \int (\int \chi^{2}(\frac{y - s}{R}) |v_{s}(t,y)|^{2}) (\int_{|x - x(t)| \leq 2C(\eta)} \chi^{2}(\frac{x_{\ast} - s}{R}) |v_{s}(t,x)|^{4} dx) ds \\
+ O(\frac{C(\eta)}{R}) \frac{1}{R^{2}} \int (\int \chi^{2}(\frac{y - s}{R}) |v_{s}(t,y)|^{2}) (\int_{|x - x(t)| \leq 2C(\eta)} |v_{s}(t,x)|^{4} dx) ds \\
=  \frac{1}{R^{2}} \int (\int \chi^{2}(\frac{y - s}{R}) |v_{s}(t,y)|^{2}) (\int_{|x - x(t)| \leq 2C(\eta)} \chi^{2}(\frac{x_{\ast} - s}{R}) |v_{s}(t,x)|^{4} dx) ds + O(\frac{C(\eta)}{R} \| v \|_{L^{2}}^{2} \| v \|_{L^{4}}^{4}).
\endaligned
\end{equation}

Plugging $(\ref{3.20})$ in to $(\ref{3.27})$, and using Strichartz estimates as in $(\ref{3.13.2})$ and $(\ref{3.13.3})$,
\begin{equation}\label{3.21}
\aligned
2 \int_{0}^{T} \frac{1}{R^{2}} \int (\int \chi^{2}(\frac{y - s}{R}) |v_{s}(t,y)|^{2}) (\int \chi^{2}(\frac{x - s}{R}) |\nabla (v_{s})(t,x)|^{2} dx) ds dt \\
- \int_{0}^{T} \frac{1}{R^{2}} \int (\int \chi^{2}(\frac{y - s}{R}) |v_{s}(t,y)|^{2}) (\int_{|x - x(t)| \leq 2C(\eta)} \chi^{2}(\frac{x_{\ast} - s}{R}) |v_{s}(t,x)|^{4} dx) ds dt \\ \lesssim R o(T) + \eta T + \frac{C(\eta)}{R} T + \frac{\eta^{4}}{R^{2}} T + \frac{\eta^{2}}{R^{2}} \int_{0}^{T} \int (\int \chi^{2}(\frac{y - s}{R}) |v_{s}|^{2} dy) (\int_{|x - x(t)| \leq 2C(\eta)} \chi^{2}(\frac{x^{\ast} - s}{R}) |v_{s}|^{4} dx) ds dt .
\endaligned
\end{equation}
Since $\chi(\frac{x^{\ast} - 1}{R}) \leq 1$, $(\ref{3.13.2})$ and $(\ref{3.13.3})$ also imply
\begin{equation}
\frac{\eta^{2}}{R^{2}} \int_{0}^{T} \int (\int \chi^{2}(\frac{y - s}{R}) |v_{s}|^{2} dy) (\int_{|x - x(t)| \leq C(\eta)} \chi^{2}(\frac{x^{\ast} - s}{R}) |v_{s}|^{4} dx) ds dt \lesssim \eta^{2} T.
\end{equation}
By definition of $x_{\ast}$ and $\chi$, for $R \gg C(\eta)$,
\begin{equation}\label{3.22}
\aligned
2 \int_{0}^{T} \frac{1}{R^{2}} \int (\int \chi^{2}(\frac{y - s}{R}) |v_{s}(t,y)|^{2}) (\int \chi^{2}(\frac{x - s}{R}) |\nabla (v_{s})(t,x)|^{2} dx) ds dt \\
- \int_{0}^{T} \frac{1}{R^{2}} \int (\int \chi^{2}(\frac{y - s}{R}) |v_{s}(t,y)|^{2}) (\int_{|x - x(t)| \leq 2C(\eta)} \chi^{2}(\frac{x_{\ast} - s}{R}) |v_{s}(t,x)|^{4} dx) ds dt \\
\geq 2 \int_{0}^{T} \frac{1}{R} \int (\int \chi^{2}(\frac{y - s}{R}) |v_{s}(t,y)|^{2}) ( \chi^{2}(\frac{x_{\ast} - s}{R}) \int \chi^{2}(\frac{x - x(t)}{R}) |\nabla (v_{s})(t,x)|^{2} dx) ds dt \\
- \int_{0}^{T} \frac{1}{R} \int (\int \chi^{2}(\frac{y - s}{R}) |v_{s}(t,y)|^{2}) (\chi^{2}(\frac{x_{\ast} - s}{R})\int \chi^{4}(\frac{x - x(t)}{R}) |v_{s}(t,x)|^{4} dx) ds dt.
\endaligned
\end{equation}
Integrating by parts,
\begin{equation}\label{3.23}
\int \chi^{2}(\frac{x - x(t)}{R}) |\nabla (v_{s})|^{2} dx = \int |\nabla(\chi(\frac{x - x(t)}{R}) v_{s})|^{2} dx + \frac{1}{R^{2}} \int \chi''(\frac{x - x(t)}{R}) \chi(\frac{x - x(t)}{R}) |v_{s}|^{2} dx.
\end{equation}
Therefore,
\begin{equation}\label{3.24}
\aligned
2 \int_{0}^{T} \frac{1}{R^{2}} \int (\int \chi^{2}(\frac{y - s}{R}) |v_{s}(t,y)|^{2}) ( \chi^{2}(\frac{x_{\ast} - s}{R}) \int \chi^{2}(\frac{x - x(t)}{R}) |\nabla(v_{s})(t,x)|^{2} dx) ds dt \\
- \int_{0}^{T} \frac{1}{R^{2}} \int (\int \chi^{2}(\frac{y - s}{R}) |v_{s}(t,y)|^{2}) (\chi^{2}(\frac{x_{\ast} - s}{R})\int \chi^{4}(\frac{x - x(t)}{R}) |v_{s}(t,x)|^{4} dx) ds dt \\
= 4 \int_{0}^{T} \frac{1}{R^{2}} \int (\int \chi^{2}(\frac{y - s}{R}) |v_{s}(t, y)|^{2} dy) \chi^{2}(\frac{x_{\ast} - s}{R}) E(\chi^{2}(\frac{x - x(t)}{R}) v) ds dt + O(\frac{T}{R^{2}}).
\endaligned
\end{equation}
Here $E$ is the energy given by $(\ref{1.7})$. Therefore, we have finally proved
\begin{equation}\label{3.25}
\aligned
4 \int_{0}^{T} \frac{1}{R^{2}} \int (\int \chi^{2}(\frac{y - s}{R}) |v_{s}(t, y)|^{2} dy) \chi^{2}(\frac{x_{\ast} - s}{R}) E(\chi^{2}(\frac{x - x(t)}{R}) v) ds dt \lesssim R o(T) + \eta T + \frac{C(\eta)}{R} T + \frac{\eta^{4}}{R^{2}} T.
\endaligned
\end{equation}

Choosing $R \nearrow \infty$ perhaps very slowly as $T \nearrow \infty$, and then $\eta \searrow 0$ sufficiently slowly, the right hand side of $(\ref{3.25})$ is bounded by $o(T)$.\medskip

On the other hand, when $|s - x(t)| \leq \frac{R}{2}$, $\chi(\frac{x_{\ast} - s}{R}) = 1$ and
\begin{equation}\label{3.28}
(\int \chi^{2}(\frac{y - s}{R}) |v_{s}(t, y)|^{2} dy) \geq \frac{1}{2} \| v \|_{L^{2}}^{2}.
\end{equation}
Therefore, since the Gagliardo--Nirenberg inequality guarantees that $E(u) \geq 0$ when $\| u \|_{L^{2}} \leq \| Q \|_{L^{2}}$, the left hand side of $(\ref{3.25})$ is bounded below by
\begin{equation}\label{3.29}
\| u_{0} \|_{L^{2}}^{2} \int_{0}^{T} \frac{1}{R^{2}} \int_{|s - x(t)| \leq \frac{R}{2}} E(\chi(\frac{x - x(t)}{R}) v_{s}) ds dt \lesssim o(T)
\end{equation}
Thus, taking a sequence $T_{n} \nearrow \infty$, $R_{n} \nearrow \infty$, $\eta_{n} \searrow 0$, there exists a sequence of times $t_{n} \in [\frac{T_{n}}{2}, T_{n}]$, $|s_{n} - x(t_{n})| \leq \frac{R_{n}}{2}$ such that
\begin{equation}\label{3.30}
E(\chi(\frac{x - s_{n}}{R_{n}}) e^{ix \cdot \xi(s_{n})} e^{i \gamma(s_{n})} P_{\leq T_{n}} u(t_{n}, x)) \rightarrow 0,
\end{equation}
\begin{equation}\label{3.31}
(1 - \chi(\frac{x - s_{n}}{R_{n}})) e^{ix \cdot \xi(s_{n})} e^{i \gamma(s_{n})} P_{\leq T_{n}} u(t_{n}, x) \rightarrow 0, \qquad \text{in} \qquad L^{2},
\end{equation}
\begin{equation}\label{3.31.1}
 (1 - P_{\geq T_{n}}) u(t_{n}, x) \rightarrow 0, \qquad \text{in} \qquad L^{2},
 \end{equation}
and
\begin{equation}\label{3.32}
\| \chi(\frac{x - s_{n}}{R_{n}}) e^{ix \cdot \xi(s_{n})} e^{i \gamma(s_{n})} P_{\leq T_{n}} u(t_{n}, x) \|_{L^{4}} \sim 1.
\end{equation}

Now by the almost periodicity of $u$, $(\ref{2.8})$, after passing to a subsequence, there exists $u_{0} \in H^{1}$ such that
\begin{equation}\label{3.33}
\chi(\frac{x + x(t_{n}) - s_{n}}{R_{n}}) e^{ix \cdot \xi(s_{n})} e^{i x(t_{n}) \cdot \xi(s_{n})} e^{i \gamma(s_{n})} P_{\leq T_{n}} u(t_{n}, x + x(t_{n})) \rightharpoonup u_{0},
\end{equation}
weakly in $H^{1}$, and
\begin{equation}\label{3.34}
\chi(\frac{x + x(t_{n}) - s_{n}}{R_{n}}) e^{ix \cdot \xi(s_{n})} e^{i x(t_{n}) \cdot \xi(s_{n})} P_{\leq T_{n}} u(t_{n}, x + x(t_{n})) \rightarrow u_{0},
\end{equation}
strongly in in $L^{2} \cap L^{4}$. Also, by $(\ref{3.30})$, $(\ref{3.31})$, and $(\ref{3.31.1})$, $\| u_{0} \|_{L^{2}} = \| Q \|_{L^{2}}$, $E(u_{0}) \leq 0$, and by the Gagliardo-Nirenberg inequality, $E(u_{0}) = 0$. Therefore,
\begin{equation}\label{3.36}
u_{0} = \lambda Q(\lambda(x - x_{0})),
\end{equation}
for some $\lambda \sim 1$ and $|x_{0}| \lesssim 1$. This proves Theorem $\ref{t2.2}$ when $\lambda(t) = 1$. $\Box$

\section{Proof of Theorem $\ref{t2.2}$ when $\lambda(t) = 1$ and $d \geq 3$}
The proof of Theorem $\ref{t2.2}$ when $\lambda(t) = 1$ in higher dimensions is quite similar to the proof in two dimensions. In this case as well, use the interaction Morawetz estimate
\begin{equation}\label{6.1}
M(t) = \int \int |Iu(t, y)|^{2} Im[\bar{Iu} \nabla Iu](t,x) \cdot (x - y) \psi(x - y) dx dy,
\end{equation}
where $I$ is the Fourier truncation operator $P_{\leq T}$, $T = 2^{k}$, $k \in \mathbb{Z}_{\geq 0}$. Again let
\begin{equation}\label{6.2}
\psi(x) = \frac{1}{|x - y|} \int_{0}^{|x - y|} \phi(s) ds,
\end{equation}
where $\phi(|x|)$ is a radial function given by
\begin{equation}\label{6.3}
\phi(|x - y|) = \frac{1}{R^{d}} \int \chi^{2}(\frac{x - y - s}{R}) \chi^{2}(\frac{s}{R}) ds = \frac{1}{R^{d}} \int \chi^{2}(\frac{x - s}{R}) \chi^{2}(\frac{s - y}{R}) ds = \frac{1}{R^{d}} \int \chi^{2}(\frac{x - s}{R}) \chi^{2}(\frac{y - s}{R}) ds,
\end{equation}
where once again $\chi$ is a radial, smooth, compactly supported function, $\chi(x) = 1$ for $|x| \leq 1$, $\chi(x)$ is supported on $|x| \leq 2$, and $\chi(|x|)$ is decreasing as a function of the radius.

By direct computation,
\begin{equation}\label{6.4}
\aligned
\frac{d}{dt} M(t) = 2 \int \int |Iu(t,y)|^{2} Re[\partial_{j} \bar{Iu} \partial_{k} Iu](t,x) [\delta_{jk} \psi(x - y) + \frac{(x - y)_{j} (x - y)_{k}}{|x - y|} \psi'(x - y)] dx dy \\ -2 \int \int Im[\bar{Iu} \partial_{k} Iu](t,y) Im[\bar{Iu} \partial_{j} Iu](t,x) [\delta_{jk} \psi(x - y) + \frac{(x - y)_{j} (x - y)_{k}}{|x - y|} \psi'(x - y)] dx dy \\
+ \frac{1}{2} \int \int |Iu(t,y)|^{2} |Iu(t,y)|^{2} [\Delta \phi(x - y) + (d - 1) \Delta \psi(x - y)] dx dy \\
- \frac{2}{d + 2} \int \int |Iu(t,y)|^{2} |Iu(t,x)|^{2 + \frac{4}{d}} [d \psi(x - y) + \psi'(x - y) |x - y|] dx dy + \mathcal E,
\endaligned
\end{equation}
where $\mathcal E$ are the error terms arising from $\mathcal N$,
\begin{equation}\label{6.5}
i Iu_{t} + \Delta I u + F(Iu) = F(Iu) - I F(u) = \mathcal N.
\end{equation}

As in the previous section, it is known from \cite{dodson2012global} that
\begin{equation}\label{6.6}
\int_{0}^{T} \mathcal N dt \lesssim R o(T),
\end{equation}
and
\begin{equation}\label{6.7}
\sup_{t \in [0, T]} |M(t)| \lesssim R o(T).
\end{equation}
Therefore, choosing $R \nearrow \infty$ sufficiently slowly,
\begin{equation}\label{6.8}
\lim_{T \rightarrow \infty} \frac{R o(T)}{T} = 0.
\end{equation}

The computations in $(\ref{3.5})$--$(\ref{3.8})$ can easily be generalized to higher dimensions. Indeed,
\begin{equation}\label{6.9}
\phi(x) = \frac{1}{R^{d}} \int \chi^{2}(\frac{x - s}{R}) \chi^{2}(\frac{s}{R}) ds \sim 1,
\end{equation}
for $|x| \leq R$, $\phi(x)$ is supported on the set $|x| \leq 4R$, and $\phi(x)$ is a radially symmetric function that is decreasing as $|x| \rightarrow \infty$. Therefore, $(\ref{6.2})$ implies that
\begin{equation}\label{6.10}
|\psi(x)| \lesssim \frac{R}{|x|}, \qquad \text{for all} \qquad x \in \mathbb{R}^{d}.
\end{equation}
Also, by direct computation,
\begin{equation}\label{6.11}
\Delta \phi(x) = \frac{1}{R^{d}} \int \Delta \chi^{2}(\frac{x - s}{R}) \chi^{2}(\frac{s}{R}) ds \lesssim \frac{1}{R^{2}}.
\end{equation}
Next, by the same calculations that give $(\ref{6.10})$,
\begin{equation}\label{6.12}
\Delta \psi(x) \lesssim \frac{R}{|x|^{3}},
\end{equation}
so $|\Delta \psi(x)| \lesssim \frac{1}{R^{2}}$ for $|x| \gtrsim R$. Also,
\begin{equation}\label{6.13}
\psi(r) = \phi(0) + \frac{1}{r} \int_{0}^{r} \int_{0}^{s} (s - t) \phi''(t) dt ds,
\end{equation}
so by $(\ref{6.11})$, $|\Delta \psi(x)| \lesssim \frac{1}{R^{2}}$ for $|x| \lesssim R$.
Therefore,
\begin{equation}\label{6.14}
 \frac{1}{2} \int \int |Iu(t,y)|^{2} |Iu(t,y)|^{2} [\Delta \phi(x - y) + \Delta \psi(x - y)] dx dy \lesssim \frac{1}{R^{2}} \| u \|_{L^{2}}^{4}.
 \end{equation}

Following the case when $d = 2$,
\begin{equation}\label{6.15}
\delta_{jk} \psi(x - y) + \frac{(x - y)_{j} (x - y)_{k}}{|x - y|} \psi'(x - y) = \delta_{jk} \phi(x - y) - \delta_{jk} |x - y| \psi'(|x - y|) + \frac{(x - y)_{j} (x - y)_{k}}{|x - y|} \psi'(x - y),
\end{equation}
and
\begin{equation}\label{6.16}
\aligned
-\int \int Im[\bar{Iu} \partial_{k} Iu] Im[\bar{Iu} \partial_{j} Iu_{x}] \delta_{jk} \phi(x - y) dx dy + \int \int |Iu(t,y)|^{2} |\nabla Iu(t,x)|^{2} \phi(x - y) dx dy \\
= \frac{1}{R^{d}} \int (\int \chi^{2}(\frac{y - s}{R}) Im[\bar{Iu} \partial_{j} Iu] dy)(\int \chi^{2}(\frac{x - s}{R}) Im[\bar{Iu} \partial_{j} Iu] dx) ds \\ + \frac{1}{R^{d}} \int (\int \chi^{2}(\frac{y - s}{R}) |Iu(t,y)|^{2} dy)(\int \chi^{2}(\frac{x - s}{R}) |\nabla Iu(t,x)|^{2} dx) ds,
\endaligned
\end{equation}
so for a fixed $s \in \mathbb{R}^{d}$, for any $\xi \in \mathbb{R}^{d}$ and $j \in \mathbb{Z}$,
\begin{equation}\label{6.17}
\int \chi^{2}(\frac{y - s}{R}) Im[\overline{e^{iy \xi} Iu} \partial_{j}(e^{iy \xi} Iu)] dy = \int \chi^{2}(\frac{y - s}{R}) Im[\bar{Iu} \partial_{j} Iu] dy + \xi_{j} \int \chi^{2}(\frac{y - s}{R}) |Iu(t,y)|^{2} dy,
\end{equation}
and
\begin{equation}\label{6.18}
\aligned
\int \chi^{2}(\frac{x - s}{R}) |\partial_{j} (e^{ix \xi} Iu)|^{2} dx = \xi_{j}^{2} \int \chi^{2}(\frac{x - s}{R}) |Iu|^{2} dx \\ + 2 \xi_{j} \int \chi^{2}(\frac{x - s}{R}) Im[\bar{Iu} \partial_{j} Iu] dx + \int \chi^{2}(\frac{x - s}{R}) |\partial_{j} Iu|^{2} dx.
\endaligned
\end{equation}
So it is again convenient to choose $\xi(s) \in \mathbb{R}^{d}$ such that the $(\ref{6.17}) = 0$ for any $j \in \mathbb{Z}$. Set
\begin{equation}\label{6.19}
v_{s} = e^{ix \xi(s)} Iu.
\end{equation}
By the fundamental theorem of calculus,
\begin{equation}\label{6.20}
\aligned
2 \int_{0}^{T} \frac{1}{R^{d}} \int (\int \chi^{2}(\frac{y - s}{R}) |v_{s}(t,y)|^{2}) (\int \chi^{2}(\frac{x - s}{R}) |\nabla (v_{s})(t,x)|^{2} dx) ds dt \\
2 \int \int |Iu(t,y)|^{2} Re[\partial_{j} \bar{Iu} \partial_{k} Iu](t,x) [\delta_{jk} |x - y| \psi'(x - y) + \frac{(x - y)_{j} (x - y)_{k}}{|x - y|} \psi'(x - y)] dx dy \\ -2 \int \int Im[\bar{Iu} \partial_{k} Iu](t,y) Im[\bar{Iu} \partial_{j} Iu](t,x) [\delta_{jk} |x - y| \psi'(x - y) + \frac{(x - y)_{j} (x - y)_{k}}{|x - y|} \psi'(x - y)] dx dy \\
- \frac{d}{d + 2} \int_{0}^{T} \frac{1}{R^{d}} \int (\int \chi^{2}(\frac{y - s}{R}) |v_{s}(t,y)|^{2}) (\int \chi^{2}(\frac{x - s}{R}) |v_{s}(t,x)|^{\frac{2(d + 2)}{d}} dx) ds dt \\
- \frac{2(d - 1)}{d + 2} \int_{0}^{T} \int |Iu(t,y)|^{2} [\psi(x - y) - \phi(x - y)] |Iu(t,x)|^{\frac{2(d + 2)}{d}} dx dy dt \lesssim R o(T) + \frac{1}{R^{2}} T.
\endaligned
\end{equation}
Again following \cite{dodson2015global} in higher dimensions,
\begin{equation}\label{6.21}
\aligned
2 \int \int |Iu(t,y)|^{2} Re[\partial_{j} \bar{Iu} \partial_{k} Iu](t,x) [\delta_{jk} |x - y| \psi'(x - y) + \frac{(x - y)_{j} (x - y)_{k}}{|x - y|} \psi'(x - y)] dx dy \\ -2 \int \int Im[\bar{Iu} \partial_{k} Iu](t,y) Im[\bar{Iu} \partial_{j} Iu](t,x) [\delta_{jk} |x - y| \psi'(x - y) + \frac{(x - y)_{j} (x - y)_{k}}{|x - y|} \psi'(x - y)] dx dy \geq 0.
\endaligned
\end{equation}
Therefore,
\begin{equation}\label{6.22}
\aligned
2 \int_{0}^{T} \frac{1}{R^{d}} \int (\int \chi^{2}(\frac{y - s}{R}) |v_{s}(t,y)|^{2}) (\int \chi^{2}(\frac{x - s}{R}) |\nabla (v_{s})(t,x)|^{2} dx) ds dt \\
- \frac{2d}{d + 2} \int_{0}^{T} \frac{1}{R^{d}} \int (\int \chi^{2}(\frac{y - s}{R}) |v_{s}(t,y)|^{2}) (\int \chi^{2}(\frac{x - s}{R}) |v_{s}(t,x)|^{2 + \frac{4}{d}} dx) ds dt \\
- \frac{2(d - 1)}{d + 2} \int_{0}^{T} \int |Iu(t,y)|^{2} [\psi(x - y) - \phi(x - y)] |Iu(t,x)|^{2 + \frac{4}{d}} dx dy dt \lesssim R o(T) + \frac{1}{R^{2}} T.
\endaligned
\end{equation}

Again by the Arzela--Ascoli theorem, H{\"o}lder's inequality, Strichartz estimates, and $\lambda(t) = 1$,
\begin{equation}\label{6.23}
\aligned
\int_{a}^{a + 1} \int_{|y - y(t)| \geq C(\eta)} |Iu(t,y)|^{2} \int |Iu(t,x)|^{2 + \frac{4}{d}} dx dy dt + \int_{a}^{a + 1} \int |Iu(t,y)|^{2} \int_{|x - x(t)| \geq C(\eta)} |Iu(t,x)|^{2 + \frac{4}{d}} dx dy dt \\ \lesssim \eta^{2} \| u \|_{L_{t,x}^{2 + \frac{4}{d}}([a, a + 1] \times \mathbb{R}^{d})}^{2 + \frac{4}{d}} + \eta^{\frac{4}{d}} \| u \|_{L_{t}^{2} L_{x}^{\frac{2d}{d + 2}}([a, a + 1] \times \mathbb{R}^{d})}^{2}.
\endaligned
\end{equation}
Finally, by $(\ref{6.2})$ and the fundamental theorem of calculus,
\begin{equation}\label{6.24}
\int_{a}^{a + 1} \int_{|y - y(t)| \leq C(\eta)} \int_{|x - x(t)| \leq C(\eta)} |Iu(t,y)|^{2} [\psi(x - y) - \phi(x - y)] |Iu(t,x)|^{2 + \frac{4}{d}} dx dy dt \lesssim \frac{C(\eta)}{R} \| u \|_{L_{t,x}^{2 + \frac{4}{d}}([a, a + 1] \times \mathbb{R}^{2})}^{2 + \frac{4}{d}}.
\end{equation}
Therefore, letting $\sigma = \inf \{ 1, \frac{4}{d} \}$,
\begin{equation}\label{6.25}
\int_{0}^{T} \int |Iu(t,y)|^{2} [\psi(x - y) - \phi(x - y)] |Iu(t,x)|^{2 + \frac{4}{d}} dx dy dt \lesssim \eta^{\sigma} T + \frac{C(\eta)}{R} T,
\end{equation}
so
\begin{equation}\label{6.26}
\aligned
2 \int_{0}^{T} \frac{1}{R^{d}} \int (\int \chi^{2}(\frac{y - s}{R}) |v_{s}(t,y)|^{2}) (\int \chi^{2}(\frac{x - s}{R}) |\nabla (v_{s})(t,x)|^{2} dx) ds dt \\
- \frac{2}{d + 2} \int_{0}^{T} \frac{1}{R^{d}} \int (\int \chi^{2}(\frac{y - s}{R}) |v_{s}(t,y)|^{2}) (\int \chi^{2}(\frac{x - s}{R}) |v_{s}(t,x)|^{\frac{2(d + 2)}{d}} dx) ds dt \lesssim R o(T) + \eta^{\sigma} T + \frac{C(\eta)}{R} T.
\endaligned
\end{equation}

In dimensions $d \geq 3$, we will use the Sobolev embedding theorem, as in \cite{merle1993determination}, to control $|u(t,x)|^{2 + \frac{4}{d}}$ far away from $x(t)$. By the product rule, for any $f \in H^{1}(\mathbb{R}^{d})$,
\begin{equation}\label{6.27}
\aligned
\int \chi(\frac{x - s}{R})^{2} |f(x)|^{2 + \frac{4}{d}} dx \lesssim \| \chi(\frac{x - s}{R}) f \|_{L^{\frac{2d}{d + 2}}(\mathbb{R}^{d})}^{2} \| f \|_{L^{2}(\mathbb{R}^{d})}^{\frac{4}{d}} \lesssim \| \nabla(\chi(\frac{x - s}{R}) f) \|_{L^{2}}^{2} \| f \|_{L^{2}}^{\frac{4}{d}} \\
\lesssim \| f \|_{L^{2}}^{\frac{4}{d}} (\| \chi(\frac{x - s}{R}) \nabla f \|_{L^{2}}^{2} + \frac{1}{R^{2}} \| f \|_{L^{2}}^{2}).
\endaligned
\end{equation}

For a fixed $t$, let $f = (1 - \chi(\frac{x - x(t)}{C(\eta)})) v_{s}(t,x)$. By $(\ref{3.13.1})$, $\| f \|_{L^{2}} \leq \eta$, and by the product rule,
\begin{equation}\label{6.28}
\| \chi(\frac{x - s}{R}) \nabla f \|_{L^{2}} \leq \| \chi(\frac{x - s}{R}) \nabla v_{s} \|_{L^{2}} + \frac{1}{C(\eta)} \| \chi'(\frac{x - x(t)}{C(\eta)}) v_{s} \|_{L^{2}} \leq \| \chi(\frac{x - s}{R}) \nabla v_{s} \|_{L^{2}} + \frac{\eta}{C(\eta)}.
\end{equation}

Therefore, since $\frac{1}{R^{d}} \int (\int \chi^{2}(\frac{y - s}{R}) |v_{s}(t,y)|^{2} dy) ds \lesssim \| u \|_{L^{2}}^{2}$,
\begin{equation}\label{6.29}
\aligned
2 \int_{0}^{T} \frac{1}{R^{d}} \int (\int \chi^{2}(\frac{y - s}{R}) |v_{s}(t,y)|^{2}) (\int \chi^{2}(\frac{x - s}{R}) |\nabla (v_{s})(t,x)|^{2} dx) ds dt \\
- \frac{d}{d + 2} \int_{0}^{T} \frac{1}{R^{d}} \int (\int \chi^{2}(\frac{y - s}{R}) |v_{s}(t,y)|^{2}) (\int_{|x - x(t)| \leq 2C(\eta)} \chi^{2}(\frac{x - s}{R}) |v_{s}(t,x)|^{2 + \frac{4}{d}} dx) ds dt \\ \lesssim R o(T) + \eta^{\sigma} T + \frac{C(\eta)}{R} T + \frac{\eta^{2 + \frac{4}{d}}}{R^{2}} T + \frac{\eta^{2}}{R^{2}} \int_{0}^{T} \int (\int \chi^{2}(\frac{y - s}{R}) |v_{s}|^{2} dy) (\int \chi^{2}(\frac{x - s}{R}) |\nabla (v_{s})|^{2} dx) ds dt.
\endaligned
\end{equation}
When $\eta > 0$ is sufficiently small,
\begin{equation}\label{6.30}
 \frac{\eta^{2}}{R^{2}} \int_{0}^{T} \int (\int \chi^{2}(\frac{y - s}{R}) |v_{s}|^{2} dy) (\int \chi^{2}(\frac{x - s}{R}) |\nabla (v_{s})|^{2} dx) ds dt
 \end{equation}
 can be absorbed into the left hand side of $(\ref{6.29})$, proving that
 \begin{equation}\label{6.31}
\aligned
2 \int_{0}^{T} \frac{1}{R} \int (\int \chi^{2}(\frac{y - s}{R}) |v_{s}(t,y)|^{2}) (\int \chi^{2}(\frac{x - s}{R}) |\nabla (v_{s})(t,x)|^{2} dx) ds dt \\
- \frac{d}{d + 2} \int_{0}^{T} \frac{1}{R} \int (\int \chi^{2}(\frac{y - s}{R}) |v_{s}(t,y)|^{2}) (\int_{|x - x(t)| \leq 2C(\eta)} \chi^{2}(\frac{x - s}{R}) |v_{s}(t,x)|^{2 + \frac{4}{d}} dx) ds dt \\ \lesssim R o(T) + \eta^{\sigma} T + \frac{C(\eta)}{R} T+  \frac{\eta^{2 + \frac{4}{d}}}{R^{2}} T \\ + \frac{ \eta^{2}}{R^{2}} \int_{0}^{T} \int (\int \chi^{2}(\frac{y - s}{R}) |v_{s}(t,y)|^{2} dy) (\int_{|x - x(t)| \leq 2C(\eta)} \chi^{2}(\frac{x - s}{R}) |v_{s}(t,x)|^{2 + \frac{4}{d}} dx) ds dt.
\endaligned
\end{equation}

The rest of the argument is identical to the $d = 2$ case. Choose $|x_{\ast} - x(t)| \leq 4 C(\eta)$ such that
\begin{equation}\label{6.32}
\chi(\frac{x_{\ast} - s}{R}) = \inf_{|x - x(t)| \leq 4 C(\eta)} \chi(\frac{x - s}{R}).
\end{equation}
For $|x - x(t)| \leq 2 C(\eta)$,
\begin{equation}\label{6.33}
\chi^{2}(\frac{x - s}{R}) = \chi^{2}(\frac{x_{\ast} - s}{R}) + O(\frac{C(\eta)}{R}).
\end{equation}
Therefore,
\begin{equation}\label{6.34}
\aligned
\frac{1}{R^{d}} \int (\int \chi^{2}(\frac{y - s}{R}) |v_{s}(t,y)|^{2}) (\int_{|x - x(t)| \leq 2C(\eta)} \chi^{2}(\frac{x - s}{R}) |v_{s}(t,x)|^{2 + \frac{4}{d}} dx) ds \\ \leq \frac{1}{R^{d}} \int (\int \chi^{2}(\frac{y - s}{R}) |v_{s}(t,y)|^{2}) (\int_{|x - x(t)| \leq 2C(\eta)} \chi^{2}(\frac{x_{\ast} - s}{R}) |v_{s}(t,x)|^{2 + \frac{4}{d}} dx) ds \\
+ O(\frac{C(\eta)}{R}) \frac{1}{R^{d}} \int (\int \chi^{2}(\frac{y - s}{R}) |v_{s}(t,y)|^{2}) (\int_{|x - x(t)| \leq 2C(\eta)} |v_{s}(t,x)|^{2 + \frac{4}{d}} dx) ds \\
=  \frac{1}{R^{d}} \int (\int \chi^{2}(\frac{y - s}{R}) |v_{s}(t,y)|^{2}) (\int_{|x - x(t)| \leq 2C(\eta)} \chi^{2}(\frac{x_{\ast} - s}{R}) |v_{s}(t,x)|^{2 + \frac{4}{d}} dx) ds + O(\frac{C(\eta)}{R} \| v \|_{L^{2}}^{2} \| v \|_{L^{2 + \frac{4}{d}}}^{2 + \frac{4}{d}}).
\endaligned
\end{equation}

Again using Strichartz estimates,
\begin{equation}\label{6.35}
\aligned
2 \int_{0}^{T} \frac{1}{R^{d}} \int (\int \chi^{2}(\frac{y - s}{R}) |v_{s}(t,y)|^{2}) (\int \chi^{2}(\frac{x - s}{R}) |\nabla (v_{s})(t,x)|^{2} dx) ds dt \\
- \frac{d}{d + 2} \int_{0}^{T} \frac{1}{R^{d}} \int (\int \chi^{2}(\frac{y - s}{R}) |v_{s}(t,y)|^{2}) (\int_{|x - x(t)| \leq 2C(\eta)} \chi^{2}(\frac{x_{\ast} - s}{R}) |v_{s}(t,x)|^{2 + \frac{4}{d}} dx) ds dt \\ \lesssim R o(T) + \eta^{\sigma} T + \frac{C(\eta)}{R} T + \frac{\eta^{2 + \frac{4}{d}}}{R^{2}} T \\ + \frac{\eta^{2}}{R^{d}} \int_{0}^{T} \int (\int \chi^{2}(\frac{y - s}{R}) |v_{s}|^{2} dy) (\int_{|x - x(t)| \leq 2C(\eta)} \chi^{2}(\frac{x^{\ast} - s}{R}) |v_{s}|^{2 + \frac{4}{d}} dx) ds dt,
\endaligned
\end{equation}
and
\begin{equation}\label{6.36}
\frac{\eta^{2}}{R^{d}} \int_{0}^{T} \int (\int \chi^{2}(\frac{y - s}{R}) |v_{s}|^{2} dy) (\int_{|x - x(t)| \leq C(\eta)} \chi^{2}(\frac{x^{\ast} - s}{R}) |v_{s}|^{2 + \frac{4}{d}} dx) ds dt \lesssim \eta^{2} T.
\end{equation}
By definition of $x_{\ast}$ and $\chi$, for $R \gg C(\eta)$,
\begin{equation}\label{6.37}
\aligned
2 \int_{0}^{T} \frac{1}{R^{d}} \int (\int \chi^{2}(\frac{y - s}{R}) |v_{s}(t,y)|^{2}) (\int \chi^{2}(\frac{x - s}{R}) |\nabla (v_{s})(t,x)|^{2} dx) ds dt \\
- \frac{d}{d + 2} \int_{0}^{T} \frac{1}{R^{d}} \int (\int \chi^{2}(\frac{y - s}{R}) |v_{s}(t,y)|^{2}) (\int_{|x - x(t)| \leq 2C(\eta)} \chi^{2}(\frac{x_{\ast} - s}{R}) |v_{s}(t,x)|^{2 + \frac{4}{d}} dx) ds dt \\
\geq 2 \int_{0}^{T} \frac{1}{R^{d}} \int (\int \chi^{2}(\frac{y - s}{R}) |v_{s}(t,y)|^{2}) ( \chi^{2}(\frac{x_{\ast} - s}{R}) \int \chi^{2}(\frac{x - x(t)}{R}) |\nabla (v_{s})(t,x)|^{2} dx) ds dt \\
- \frac{d}{d + 2} \int_{0}^{T} \frac{1}{R^{d}} \int (\int \chi^{2}(\frac{y - s}{R}) |v_{s}(t,y)|^{2}) (\chi^{2}(\frac{x_{\ast} - s}{R})\int \chi^{2 + \frac{4}{d}}(\frac{x - x(t)}{R}) |v_{s}(t,x)|^{2 + \frac{4}{d}} dx) ds dt.
\endaligned
\end{equation}
Integrating by parts,
\begin{equation}\label{6.38}
\int \chi^{2}(\frac{x - x(t)}{R}) |\nabla (v_{s})|^{2} dx = \int |\nabla(\chi(\frac{x - x(t)}{R}) v_{s})|^{2} dx + \frac{1}{R^{2}} \int \chi''(\frac{x - x(t)}{R}) \chi(\frac{x - x(t)}{R}) |v_{s}|^{2} dx.
\end{equation}
Therefore,
\begin{equation}\label{6.39}
\aligned
2 \int_{0}^{T} \frac{1}{R^{d}} \int (\int \chi^{2}(\frac{y - s}{R}) |v_{s}(t,y)|^{2}) ( \chi^{2}(\frac{x_{\ast} - s}{R}) \int \chi^{2}(\frac{x - x(t)}{R}) |\nabla(v_{s})(t,x)|^{2} dx) ds dt \\
- \frac{d}{d + 2} \int_{0}^{T} \frac{1}{R^{d}} \int (\int \chi^{2}(\frac{y - s}{R}) |v_{s}(t,y)|^{2}) (\chi^{2}(\frac{x_{\ast} - s}{R})\int \chi^{2 + \frac{4}{d}}(\frac{x - x(t)}{R}) |v_{s}(t,x)|^{2 + \frac{4}{d}} dx) ds dt \\
= 4 \int_{0}^{T} \frac{1}{R^{d}} \int (\int \chi^{2}(\frac{y - s}{R}) |v_{s}(t, y)|^{2} dy) \chi^{2}(\frac{x_{\ast} - s}{R}) E(\chi(\frac{x - x(t)}{R}) v) ds dt + O(\frac{T}{R^{2}}).
\endaligned
\end{equation}
Here $E$ is the energy given by $(\ref{1.7})$. Therefore, we have finally proved
\begin{equation}\label{6.40}
\aligned
4 \int_{0}^{T} \frac{1}{R^{d}} \int (\int \chi^{2}(\frac{y - s}{R}) |v_{s}(t, y)|^{2} dy) \chi^{2}(\frac{x_{\ast} - s}{R}) E(\chi^{2}(\frac{x - x(t)}{R}) v_{s}) ds dt \\ \lesssim R o(T) + \eta^{\sigma} T + \frac{C(\eta)}{R} T + \frac{\eta^{2 + \frac{4}{d}}}{R^{2}} T.
\endaligned
\end{equation}

Choosing $R \nearrow \infty$ perhaps very slowly as $T \nearrow \infty$, and then $\eta \searrow 0$ sufficiently slowly, the right hand side of $(\ref{6.40})$ is bounded by $o(T)$.\medskip

Again, when $|s - x(t)| \leq \frac{R}{2}$, $\chi(\frac{x_{\ast} - s}{R}) = 1$ and
\begin{equation}\label{6.41}
(\int \chi^{2}(\frac{y - s}{R}) |v_{s}(t, y)|^{2} dy) \geq \frac{1}{2} \| u \|_{L^{2}}^{2}.
\end{equation}
Therefore, since the Gagliardo--Nirenberg inequality guarantees that $E(u) \geq 0$ when $\| u \|_{L^{2}} \leq \| Q \|_{L^{2}}$, the left hand side of $(\ref{6.41})$ is bounded below by
\begin{equation}\label{6.42}
\| u_{0} \|_{L^{2}}^{2} \int_{0}^{T} \frac{1}{R^{d}} \int_{|s - x(t)| \leq \frac{R}{2}} E(\chi(\frac{x - x(t)}{R}) v_{s}) ds dt \lesssim o(T)
\end{equation}
Thus, taking a sequence $T_{n} \nearrow \infty$, $R_{n} \nearrow \infty$, $\eta_{n} \searrow 0$, there exists a sequence of times $t_{n} \in [\frac{T_{n}}{2}, T_{n}]$, $|s_{n} - x(t_{n})| \leq \frac{R_{n}}{2}$ such that
\begin{equation}\label{6.43}
E(\chi(\frac{x - s_{n}}{R_{n}}) e^{ix \cdot \xi(s_{n})} e^{i \gamma(s_{n})} P_{\leq T_{n}} u(t_{n}, x)) \rightarrow 0,
\end{equation}
\begin{equation}\label{6.44}
(1 - \chi(\frac{x - s_{n}}{R_{n}})) e^{ix \cdot \xi(s_{n})} e^{i \gamma(s_{n})} P_{\leq T_{n}} u(t_{n}, x) \rightarrow 0, \qquad \text{in} \qquad L^{2},
\end{equation}
\begin{equation}\label{6.45}
 (1 - P_{\geq T_{n}}) u(t_{n}, x) \rightarrow 0, \qquad \text{in} \qquad L^{2},
 \end{equation}
and
\begin{equation}\label{6.46}
\| \chi(\frac{x - s_{n}}{R_{n}}) e^{ix \cdot \xi(s_{n})} e^{i \gamma(s_{n})} P_{\leq T_{n}} u(t_{n}, x) \|_{L^{2 + \frac{4}{d}}} \sim 1.
\end{equation}

Now by the almost periodicity of $u$, $(\ref{2.8})$, after passing to a subsequence, there exists $u_{0} \in H^{1}$ such that
\begin{equation}\label{6.47}
\chi(\frac{x + x(t_{n}) - s_{n}}{R_{n}}) e^{ix \cdot \xi(s_{n})} e^{i x(t_{n}) \cdot \xi(s_{n})} e^{i \gamma(s_{n})} P_{\leq T_{n}} u(t_{n}, x + x(t_{n})) \rightharpoonup u_{0},
\end{equation}
weakly in $H^{1}$, and
\begin{equation}\label{6.48}
\chi(\frac{x + x(t_{n}) - s_{n}}{R_{n}}) e^{ix \cdot \xi(s_{n})} e^{i x(t_{n}) \cdot \xi(s_{n})} P_{\leq T_{n}} u(t_{n}, x + x(t_{n})) \rightarrow u_{0},
\end{equation}
strongly in in $L^{2} \cap L^{2 + \frac{4}{d}}$. Also, by $(\ref{6.43})$, $(\ref{6.44})$, and $(\ref{6.45})$, $\| u_{0} \|_{L^{2}} = \| Q \|_{L^{2}}$, $E(u_{0}) \leq 0$, and by the Gagliardo-Nirenberg inequality, $E(u_{0}) = 0$. Therefore,
\begin{equation}\label{6.49}
u_{0} = \lambda^{d/2} Q(\lambda(x - x_{0})),
\end{equation}
for some $\lambda \sim 1$ and $|x_{0}| \lesssim 1$. This proves Theorem $\ref{t2.2}$ when $\lambda(t) = 1$. $\Box$

\section{Proof of Theorem $\ref{t2.2}$ for a general $\lambda(t)$}
Now suppose that $\lambda(t)$ is free to vary. Recall that $|\lambda'(t)| \lesssim \lambda(t)^{3}$. In this case, 
\begin{equation}\label{4.1}
\lambda(t) : I \rightarrow (0, \infty),
\end{equation}
where $I$ is the maximal interval of existence of an almost periodic solution to $(\ref{1.1})$.
\begin{theorem}\label{t4.1}
Suppose $T_{n} \in I$, $T_{n} \rightarrow \sup(I)$ is a sequence of times in $I$. Then
\begin{equation}\label{4.2}
\lim_{T_{n} \rightarrow \sup(I)} \frac{1}{\sup_{t \in [0, T_{n}]} \lambda(t)} \cdot \int_{0}^{T_{n}} \lambda(t)^{3} dt = +\infty.
\end{equation}
\end{theorem}
\begin{proof}
 Suppose that this were not true, that is, there exists a constant $C_{0} < \infty$ and a sequence $T_{n} \rightarrow \sup(I)$ such that for all $n \in \mathbb{Z}_{\geq 0}$,
\begin{equation}\label{4.3}
\frac{1}{\sup_{t \in [0, T_{n}]} \lambda(t)} \int_{0}^{T_{n}} \lambda(t)^{3} dt \leq C_{0}.
\end{equation}
This would correspond to the rapid cascade scenario in \cite{dodson2015global}, \cite{dodson2016global}, \cite{fan20182}. In those papers $N(t)$ was used instead of $\lambda(t)$. As in those papers, $\lambda(t)$ can be chosen to be continuous, so for each $T_{n}$ choose $t_{n} \in [0, T_{n}]$ such that
\begin{equation}\label{4.4}
\lambda(t_{n}) = \sup_{t \in [0, T_{n}]} \lambda(t).
\end{equation}

Since $I$ is the maximal interval of existence of $u$,
\begin{equation}\label{4.5}
\lim_{n \rightarrow \infty} \| u \|_{L_{t,x}^{\frac{2(d + 2)}{d}}([0, T_{n}] \times \mathbb{R}^{d})} = \infty.
\end{equation}
By the almost periodicity property of $u$ and $(\ref{2.8})$, there exist $x(t_{n})$, $\xi(t_{n})$, and $\gamma(t_{n})$ such that if
\begin{equation}\label{4.7}
e^{i \gamma(t_{n})} \lambda(t_{n})^{d/2} e^{ix \cdot \xi(t_{n})} e^{i \gamma(t_{n})} u(t_{n}, \lambda(t_{n}) x + x(t_{n})) = v_{n}(x),
\end{equation}
then $v_{n}$ converges to some $u_{0}$ in $L^{2}(\mathbb{R}^{d})$, where $u_{0}$ is the initial data for a solution $u$ to $(\ref{1.1})$ that blows up in both time directions, $\lambda(t) \leq 1$ for all $t \leq 0$, and
\begin{equation}\label{4.8}
\int_{-\infty}^{0} \lambda(t)^{3} dt \leq C_{0}.
\end{equation}
Following the proof in \cite{dodson2012global} in dimensions $d \geq 3$ and \cite{dodson2016global2} in dimension $d = 2$,
\begin{equation}\label{4.9}
\| u \|_{L_{t}^{\infty} \dot{H}^{s}((-\infty, 0] \times \mathbb{R}^{d})} \lesssim_{s} C_{0}^{s},
\end{equation}
for any $0 \leq s < 1 + \frac{4}{d}$. Combining $(\ref{4.9})$ with $(\ref{4.8})$ and $|\lambda'(t)| \lesssim \lambda(t)^{3}$ implies
\begin{equation}\label{4.10}
\lim_{t \searrow -\infty} \lambda(t) = 0.
\end{equation}
Also, since
\begin{equation}\label{4.11}
|\xi'(t)| \lesssim \lambda(t)^{3},
\end{equation}
Equation $(\ref{4.8})$ implies that $\xi(t)$ converges to some $\xi_{-} \in \mathbb{R}^{d}$ as $t \searrow -\infty$. Make a Galilean transformation so that $\xi_{-} = 0$. Then, by interpolation, $(\ref{4.9})$ and $(\ref{4.10})$ imply
\begin{equation}\label{4.12}
\lim_{t \searrow -\infty} E(u(t)) = 0.
\end{equation}

Therefore, by conservation of energy, and convergence in $L^{2}$ of $(\ref{4.7})$,
\begin{equation}\label{4.13}
E(u_{0}) = 0, \qquad \text{and} \qquad \| u_{0} \|_{L^{2}} = \| Q \|_{L^{2}}.
\end{equation}
Therefore, by the Gagliardo-Nirenberg theorem,
\begin{equation}\label{4.14}
u_{0} = \lambda^{d/2} Q(\lambda(x - x_{0})), \qquad 0 < \lambda < \infty, \qquad x_{0} \in \mathbb{R}^{d},
\end{equation}
and $Q$ is the solution to the elliptic partial differential equation
\begin{equation}\label{4.15}
\Delta Q + |Q|^{\frac{4}{d}} Q = Q.
\end{equation}
However, assuming without loss of generality that $x_{0} = 0$ and $\lambda = 1$, the solution to $(\ref{1.1})$ is given by
\begin{equation}\label{4.16}
u(t,x) = e^{it} Q(x), \qquad t \in \mathbb{R}.
\end{equation}
However, such a solution definitely does not satisfy $(\ref{4.3})$, which gives a contradiction.

In the case that $\| u_{0} \|_{L^{2}} > \| Q \|_{L^{2}}$, Theorem $\ref{t5.2}$ implies that such a solution must blow up in finite time in both time directions, which contradicts $(\ref{4.8})$.
\end{proof}

Therefore, consider the case when
\begin{equation}\label{4.17}
\lim_{n \rightarrow \infty} \frac{1}{\sup_{t \in [0, T_{n}]} \lambda(t)} \int_{0}^{T_{n}} \lambda(t)^{3} dt = \infty.
\end{equation}
Passing to a subsequence, suppose
\begin{equation}\label{4.18}
\frac{1}{\sup_{t \in [0, T_{n}]} \lambda(t)} \int_{0}^{T_{n}} \lambda(t)^{3} dt = 2^{2n}.
\end{equation}

Then as in \cite{dodson2015global}, replace $M(t)$ in the previous section with,
\begin{equation}\label{4.19}
M(t) = \int \int |Iu(t, y)|^{2} Im[\bar{Iu} \nabla Iu](t,x) \cdot \tilde{\lambda}(t) (x - y) \psi(\tilde{\lambda}(t) (x - y)) dx dy,
\end{equation}
where $\tilde{\lambda}(t)$ is given by the smoothing algorithm from \cite{dodson2015global}. Then
\begin{equation}\label{4.20}
\aligned
\frac{d}{dt} M(t) = -2 \tilde{\lambda}(t) \int \int Im[\bar{Iu} \partial_{k} Iu](t,y) Im[\bar{Iu} \partial_{j} Iu](t,x) \\ \times [\delta_{jk} \psi(\tilde{\lambda}(t)(x - y)) + \frac{(x - y)_{j} (x - y)_{k}}{|x - y|} \psi'(\tilde{\lambda}(t)(x - y))] dx dy \\
+ \frac{1}{2} \tilde{\lambda}(t)^{3} \int \int |Iu(t,y)|^{2} |Iu(t,y)|^{2} [\Delta \phi(\tilde{\lambda}(t) (x - y)) + (d - 1) \Delta \psi(\tilde{\lambda}(t) (x - y))]dx dy \\
+ 2 \tilde{\lambda}(t) \int \int |Iu(t,y)|^{2} Re[\partial_{k} \bar{Iu} \partial_{j} Iu](t,x) [\delta_{jk} \psi(\tilde{\lambda}(t) (x - y)) + \frac{(x - y)_{j} (x - y)_{k}}{|x - y|} \psi'(\tilde{\lambda}(t) (x - y))] dx dy \\
- \frac{2}{d + 2} \tilde{\lambda}(t) \int \int |Iu(t,y)|^{2} |Iu(t,x)|^{2 + \frac{4}{d}} [d \psi(\tilde{\lambda}(t)(x - y)) + \psi'(\tilde{\lambda}(t) (x - y)) |x - y|]  dx dy + \mathcal E \\
+ \dot{\tilde{\lambda}}(t) \int \int |Iu(t, y)|^{2} Im[\bar{Iu} \nabla Iu](t,x) \cdot \phi(\tilde{\lambda}(t) (x - y)) (x - y) dx dy,
\endaligned
\end{equation}
where $I = P_{\leq 2^{2n} \cdot \sup_{t \in [0, T]} \lambda(t)}$.

Equations $(\ref{3.6})$ and $(\ref{6.10})$ imply
\begin{equation}\label{4.21}
\sup_{t \in [0, T_{n}]} |M(t)| \lesssim R o(2^{2n}) \cdot \sup_{t \in [0, T]} \lambda(t).
\end{equation}
Next, since the smoothing algorithm guarantees that $\tilde{\lambda}(t) \leq \lambda(t)$, following $(\ref{3.8})$ and $(\ref{6.14})$,
\begin{equation}\label{4.22}
\aligned
\int_{0}^{T_{n}} \frac{1}{2} \tilde{\lambda}(t)^{3} \int \int |Iu(t,y)|^{2} |Iu(t,y)|^{2}  [\Delta \phi(\tilde{\lambda}(t) (x - y)) + (d - 1) \Delta \psi(\tilde{\lambda}(t) (x - y))] dx dy dt \\ \lesssim \frac{1}{R^{2}} \| u \|_{L^{2}}^{4} \cdot \int_{0}^{T_{n}} \tilde{\lambda}(t) \lambda(t)^{2} dt \lesssim \frac{2^{2n}}{R^{2}} \cdot \sup_{t \in [0, T]} \lambda(t).
\endaligned
\end{equation}
Since $\tilde{\lambda}(t) \leq \lambda(t)$, following the analysis in $(\ref{3.9})$--$(\ref{3.25})$ in two dimensions and $(\ref{6.15})$--$(\ref{6.40})$ in higher dimensions,
\begin{equation}\label{4.27}
\aligned
2 \int_{0}^{T_{n}} \tilde{\lambda}(t) \int \int Im[I\bar{u} \partial_{k} Iu] Im[\bar{Iu} \partial_{j} Iu]  [\delta_{jk} \psi(\tilde{\lambda}(t)(x - y)) + \frac{(x - y)_{j} (x - y)_{k}}{|x - y|} \psi'(\tilde{\lambda}(t)(x - y))]  dx dy dt \\
+ 2 \int_{0}^{T_{n}} \tilde{\lambda}(t) \int \int |Iu(t,y)|^{2} Re[\partial_{j} \bar{Iu} \partial_{k} Iu](t,x)   [\delta_{jk} \psi(\tilde{\lambda}(t)(x - y)) + \frac{(x - y)_{j} (x - y)_{k}}{|x - y|} \psi'(\tilde{\lambda}(t)(x - y))]  dx dy dt \\
- \frac{2}{d + 2} \int_{0}^{T_{n}} \tilde{\lambda}(t) \int \int |u(t,y)|^{2} |u(t,x)|^{2 + \frac{4}{d}}  [d \psi(\tilde{\lambda}(t)(x - y)) + \psi'(\tilde{\lambda}(t) (x - y)) |x - y|] dx dy dt
\endaligned
\end{equation}

\begin{equation}\label{4.28}
\aligned
= 4 \int_{0}^{T} \frac{\tilde{\lambda}(t) \lambda(t)^{2}}{R} \int (\int \chi^{2}(\frac{y - s}{R}) |v_{s,t}(t, y)|^{2} dy) \chi^{2}(\frac{x_{\ast} - s}{R}) E(\chi^{2}(\frac{x - x(t)}{R}) v_{s,t}(t,x)) ds dt \\ + R o(2^{2n}) \cdot \sup_{t \in [0, T]} \lambda(t) + O(\eta^{\sigma} \| u \|_{L_{t}^{\infty} L_{x}^{2}}^{2} \int_{0}^{T_{n}} \tilde{\lambda}(t) \| u(t) \|_{L^{2 + \frac{4}{d}}}^{2 + \frac{4}{d}} dt) + O(\frac{C(\eta)}{R} \| u \|_{L_{t}^{\infty} L_{x}^{2}}^{2} \int_{0}^{T_{n}} \tilde{\lambda}(t) \| u(t) \|_{L_{x}^{2 + \frac{4}{d}}}^{2 + \frac{4}{d}} dt).
\endaligned
\end{equation}

\begin{remark}
 The term $v_{s, t}$ is an abbreviation for
\begin{equation}\label{4.29}
v_{s, t} = \frac{e^{i x \xi(s)}}{\lambda(t)^{d/2}} Iu(t, \frac{x}{\lambda(t)}),
\end{equation}
where $\xi(s) \in \mathbb{R}^{d}$ is chosen such that
\begin{equation}\label{4.30}
\int \chi^{2}(\frac{\tilde{\lambda}(t) (x - s)}{R \lambda(t)}) Im[\bar{v}_{s, t} \nabla (v_{s,t})] dx = 0.
\end{equation}
\end{remark}

The error estimates can be handled in a manner similar to the previous section, see \cite{dodson2015global}. Therefore, it only remains to consider the contribution of the term in $(\ref{4.20})$ with $\dot{\tilde{\lambda}}(t)$. By direct computation,
\begin{equation}\label{4.31}
\aligned
\dot{\tilde{\lambda}}(t) \int \int |Iu(t, y)|^{2} Im[\bar{Iu} \nabla Iu](t,x) \cdot \phi(\tilde{\lambda}(t) (x - y)) (x - y) dx dy  \\
= \frac{\dot{\tilde{\lambda}}(t)}{R^{d} \tilde{\lambda}(t)} \int (\int \chi^{2}(\frac{\tilde{\lambda}(t) y - s}{R}) |Iu(t,y)|^{2} dy)(\int \chi^{2}(\frac{\tilde{\lambda}(t) x - s}{R}) Im[\bar{Iu} \nabla Iu](t,x) \cdot (x \tilde{\lambda}(t) - s) dx) ds \\
- \frac{\dot{\tilde{\lambda}}(t)}{R^{d} \tilde{\lambda}(t)} \int (\int \chi^{2}(\frac{\tilde{\lambda}(t) y - s}{R}) (y \tilde{\lambda}(t) - s) |Iu(t,y)|^{2} dy) \cdot (\int \chi^{2}(\frac{\tilde{\lambda}(t) x - s}{R}) Im[\bar{Iu} \nabla Iu] dx) ds.
\endaligned
\end{equation}
Now rescale,
\begin{equation}\label{4.32}
\aligned
= \frac{\dot{\tilde{\lambda}}(t)}{R^{d} \tilde{\lambda}(t)} \lambda(t) \int (\int \chi^{2}(\frac{\tilde{\lambda}(t) y - \lambda(t) s}{R \lambda(t)}) |\frac{1}{\lambda(t)^{d/2}} Iu(t,\frac{y}{\lambda(t)})|^{2} dy) \\
\times (\int \chi^{2}(\frac{\tilde{\lambda}(t) x - \lambda(t) s}{R \lambda(t)}) Im[\frac{1}{\lambda(t)^{d/2}} \bar{Iu}(t, \frac{x}{\lambda(t)}) \nabla (\frac{1}{\lambda(t)^{d/2}} Iu(t, \frac{x}{\lambda(t)})] \cdot (\frac{x \tilde{\lambda}(t) - s \lambda(t)}{\lambda(t)}) dx) ds \\
- \frac{\dot{\tilde{\lambda}}(t)}{R^{d} \tilde{\lambda}(t)} \lambda(t) \int (\int \chi^{2}(\frac{\tilde{\lambda}(t) y - \lambda(t) s}{R \lambda(t)}) (\frac{y \tilde{\lambda}(t) - s \lambda(t)}{\lambda(t)}) |\frac{1}{\lambda(t)^{d/2}} Iu(t,\frac{y}{\lambda(t)})|^{2} dy) \\ 
\cdot (\int \chi^{2}(\frac{\tilde{\lambda}(t) x - s \lambda(t)}{R \lambda(t)}) Im[\frac{1}{\lambda(t)^{d/2}} \bar{Iu}(t, \frac{x}{\lambda(t)}) \nabla (\frac{1}{\lambda(t)^{d/2}} Iu(t, \frac{x}{\lambda(t)}))] dx) ds.
\endaligned
\end{equation}
\begin{remark}
Throughout these calculations, we understand that $\lambda^{-d/2} Iu(\frac{x}{\lambda})$ refers to the rescaling of the function $Iu(x)$, not the $I$-operator acting on a rescaling of $u$.
\end{remark}

For any $\xi \in \mathbb{R}^{d}$,
\begin{equation}\label{4.33}
\aligned
= \frac{\dot{\tilde{\lambda}}(t)}{R^{d} \tilde{\lambda}(t)} \lambda(t) \int (\int \chi^{2}(\frac{\tilde{\lambda}(t) y - \lambda(t) s}{R \lambda(t)}) |\frac{e^{ix \cdot \xi}}{\lambda(t)^{d/2}} Iu(t,\frac{y}{\lambda(t)})|^{2} dy) \\
\times (\int \chi^{2}(\frac{\tilde{\lambda}(t) x - \lambda(t) s}{R \lambda(t)}) Im[\frac{e^{-ix \cdot \xi}}{\lambda(t)^{d/2}} \bar{Iu}(t, \frac{x}{\lambda(t)}) \nabla (\frac{e^{ix \cdot \xi}}{\lambda(t)^{d/2}} Iu(t, \frac{x}{\lambda(t)})] (\frac{x \tilde{\lambda}(t) - s \lambda(t)}{\lambda(t)}) dx) ds \\
- \frac{\dot{\tilde{\lambda}}(t)}{R^{d} \tilde{\lambda}(t)} \lambda(t) \int (\int \chi^{2}(\frac{\tilde{\lambda}(t) y - \lambda(t) s}{R \lambda(t)}) (\frac{y \tilde{\lambda}(t) - s \lambda(t)}{\lambda(t)}) |\frac{e^{ix \cdot \xi}}{\lambda(t)^{d/2}} Iu(t,\frac{y}{\lambda(t)})|^{2} dy) \\ 
\times (\int \chi^{2}(\frac{\tilde{\lambda}(t) x - s \lambda(t)}{R \lambda(t)}) Im[\frac{e^{-ix \cdot \xi}}{\lambda(t)^{d/2}} \bar{Iu}(t, \frac{x}{\lambda(t)}) \nabla(\frac{e^{ix \cdot \xi}}{\lambda(t)^{d/2}} Iu(t, \frac{x}{\lambda(t)}))] dx) ds.
\endaligned
\end{equation}
In particular, if we choose $\xi = \xi(s)$,
\begin{equation}\label{4.34}
\aligned
= \frac{\dot{\tilde{\lambda}}(t)}{R^{d} \tilde{\lambda}(t)} \lambda(t) \int (\int \chi^{2}(\frac{\tilde{\lambda}(t) y - \lambda(t) s}{R \lambda(t)}) |v_{s,t}|^{2} dy) \\ \times (\int \chi^{2}(\frac{\tilde{\lambda}(t) x - \lambda(t) s}{R \lambda(t)}) Im[\bar{v}_{s,t}(t, \frac{x}{\lambda(t)}) \nabla (v_{s,t})] \cdot (\frac{x \tilde{\lambda}(t) - s \lambda(t)}{\lambda(t)}) dx) ds \\
= \frac{\dot{\tilde{\lambda}}(t)}{R} \int (\int \chi^{2}(\frac{\tilde{\lambda}(t) (y - s)}{R \lambda(t)}) |v_{s,t}|^{2} dy)  (\int \chi^{2}(\frac{\tilde{\lambda}(t) (x - s)}{R \lambda(t)}) Im[\bar{v}_{s,t}(t, \frac{x}{\lambda(t)}) \nabla (v_{s,t})] \cdot (\frac{\tilde{\lambda}(t)(x - s)}{\lambda(t)}) dx) ds.
\endaligned
\end{equation}
Then by the Cauchy-Schwarz inequality,
\begin{equation}\label{4.35}
\aligned
\lesssim  \frac{\eta^{4}}{R^{d}} \lambda(t) \tilde{\lambda}(t)^{2} \int  (\int \chi^{2}(\frac{\tilde{\lambda}(t) (y - s)}{R \lambda(t)}) |v_{s,t}|^{2} dy) (\int \chi^{2}(\frac{\tilde{\lambda}(t) (x - s)}{R \lambda(t)}) |\partial_{x}(v_{s,t})|^{2} dx) ds \\
+ \frac{1}{\eta^{4}} \frac{|\dot{\tilde{\lambda}}(t)|^{2}}{\lambda(t) \tilde{\lambda}(t)^{2}} \frac{\tilde{\lambda}(t)}{R^{d} \lambda(t)} \int (\int \chi^{2}(\frac{\tilde{\lambda}(t) (y - s)}{R \lambda(t)}) |v_{s,t}|^{2} dy) \int \chi^{2}(\frac{\tilde{\lambda}(t) (x - s)}{R \lambda(t)}) |v_{s,t}|^{2} (\frac{\tilde{\lambda}(t)(x - s)}{\lambda(t)})^{2} dx) ds.
\endaligned
\end{equation}

The first term in $(\ref{4.35})$ can be absorbed into $(\ref{4.28})$. The second term in $(\ref{4.35})$ is bounded by
\begin{equation}\label{4.36}
\frac{1}{\eta^{4}} \frac{|\dot{\tilde{\lambda}}(t)|^{2}}{\lambda(t) \tilde{\lambda}(t)^{2}} R^{2} \| u \|_{L_{t}^{\infty} L_{x}^{2}}^{4}.
\end{equation}
The smoothing algorithm from \cite{dodson2015global} is used to control this term. Recall that after $n$ iterations of the smoothing algorithm on an interval $[0, T]$, $\tilde{\lambda}(t)$ has the following properties:

\begin{enumerate}
\item $\tilde{\lambda}(t) \leq \lambda(t)$,

\item If $\dot{\tilde{\lambda}}(t) \neq 0$, then $\lambda(t) = \tilde{\lambda}(t)$,

\item $\tilde{\lambda}(t) \geq 2^{-n} \lambda(t)$,
 
\item $\int_{0}^{T} |\dot{\tilde{\lambda}}(t)| dt \leq \frac{1}{n} \int_{0}^{T} |\dot{\lambda}(t)| \frac{\tilde{\lambda}(t)}{\lambda(t)} dt$, with implicit constant independent of $n$ and $T$.
\end{enumerate}

Therefore,
\begin{equation}\label{4.37}
\int_{0}^{T_{n}} \frac{1}{\eta^{4}} \frac{|\dot{\tilde{\lambda}}(t)|^{2}}{\lambda(t) \tilde{\lambda}(t)^{2}} R^{2} \| u \|_{L_{t}^{\infty} L_{x}^{2}}^{4} dt \leq \frac{1}{\eta^{4}} \| u \|_{L_{t}^{\infty} L_{x}^{2}}^{4} \int_{0}^{T_{n}} \frac{|\dot{\lambda}(t)|}{\lambda(t)^{3}} R^{2} |\dot{\tilde{\lambda}}(t)| dt \lesssim \frac{1}{n} \frac{R^{2}}{\eta^{4}} \| u \|_{L_{t}^{\infty} L_{x}^{2}}^{4} \int_{0}^{T_{n}} \tilde{\lambda}(t)  \lambda(t)^{2} dt.
\end{equation}

Since $\sup_{t \in [0, T_{n}]} \lambda(t) \leq 2^{-2n} \int_{0}^{T_{n}} \lambda(t)^{3} dt$,
\begin{equation}\label{4.39}
R_{n} \sup_{t \in [0, T_{n}]} |M(t)| \lesssim R_{n} o(2^{2n}) \cdot \sup_{t \in [0, T]} \lambda(t).
\end{equation}
Therefore, it is possible to take a sequence $\eta_{n} \searrow 0$, $R_{n} \nearrow \infty$, probably very slowly, such that
\begin{equation}\label{4.38}
\frac{1}{n} \frac{R_{n}^{2}}{\eta^{4}} \| u \|_{L_{t}^{\infty} L_{x}^{2}}^{4} \int_{0}^{T_{n}} |\tilde{\lambda}(t)| \lambda(t)^{2} dt = o_{n}(1) \int_{0}^{T_{n}} \tilde{\lambda}(t) \lambda(t)^{2} dt,
\end{equation}
\begin{equation}\label{4.40}
R_{n} \sup_{t \in [0, T_{n}]} |M(t)| \lesssim o(2^{2n}) \cdot \sup_{t \in [0, T]} \lambda(t),
\end{equation}
\begin{equation}\label{4.40.1}
O(\eta_{n}^{4} \| u \|_{L_{t}^{\infty} L_{x}^{2}}^{2} \int_{0}^{T} \tilde{\lambda}(t) \| u(t) \|_{L^{2 + \frac{4}{d}}}^{2 + \frac{4}{d}} dt) \lesssim o_{n}(1) \int_{0}^{T_{n}} \tilde{\lambda}(t) \lambda(t)^{2} dt,
\end{equation}
and
\begin{equation}\label{4.40.2}
O(\frac{C(\eta_{n})}{R_{n}} \| u \|_{L_{t}^{\infty} L_{x}^{2}}^{2} \int_{0}^{T_{n}} \tilde{\lambda}(t) \| u(t) \|_{L_{x}^{2 + \frac{4}{d}}}^{2 + \frac{4}{d}} dt) \lesssim o_{n}(1) \int_{0}^{T_{n}} \tilde{\lambda}(t) \lambda(t)^{2} dt.
\end{equation}

Therefore, these terms may be safely treated as error terms, and repeating the analysis in sections three and four for $(\ref{4.28})$, there exists a sequence of times $t_{n} \nearrow \sup(I)$ such that

\begin{equation}
E(\chi(\frac{(x - x(t_{n})) \tilde{\lambda}(t_{n})}{R_{n} \lambda(t_{n})}) v_{s_{n}, t_{n}}) \rightarrow 0,
\end{equation}
\begin{equation}
\| (1 - \chi(\frac{(x - x(t_{n})) \tilde{\lambda}(t_{n})}{R_{n} \lambda(t_{n})})) v_{s_{n}, t_{n}} \|_{L^{2}} \rightarrow 0,
\end{equation}
\begin{equation}
\| v_{s_{n}, t_{n}} \|_{L^{2}} \nearrow \| Q \|_{L^{2}},
\end{equation}
and
\begin{equation}\label{3.32}
\| \chi(\frac{(x - x(t_{n})) \tilde{\lambda}(t_{n})}{R_{n} \lambda(t_{n})}) Iv_{s_{n}, t_{n}} \|_{L^{2 + \frac{4}{d}}} \sim 1.
\end{equation}

In this case as well, we can show that this sequence converges in $H^{1}$ to
\begin{equation}\label{3.36}
u_{0} = \lambda^{d/2} Q(\lambda(x - x_{0})).
\end{equation}
This proves Theorem $\ref{t2.2}$ for a general $\lambda(t)$. $\Box$

\section{Proof of Theorem $\ref{t1.3}$:} 
The proof of Theorem $\ref{t1.3}$ uses the argument used in the proof of Theorem $\ref{t1.2}$, combined with some reductions from \cite{fan20182}. First recall Lemma $4.2$ from \cite{fan20182}.
\begin{lemma}\label{l5.1}
Let $u$ be a solution to $(\ref{1.1})$ that satisfies the assumptions of Theorem $\ref{t1.3}$. Then there exists a sequence $t_{n} \nearrow T^{+}(u)$ such that $u(t_{n})$ admits a profile decomposition with profiles $\{ \phi_{j}, \{ x_{j,n}, \lambda_{j,n}, \xi_{j,n}, t_{j,n}, \gamma_{j,n} \} \}$, and there is a unique profile, call it $\phi_{1}$, such that
\begin{enumerate}
\item $\| \phi_{1} \|_{L^{2}} \geq \| Q \|_{L^{2}}$,

\item The nonlinear profile $\Phi_{1}$ associated to $\phi_{1}$ is an almost periodic solution in the sense of $(\ref{2.8})$ that does not scatter forward or backward in time.
\end{enumerate}
\end{lemma}

Now consider the nonlinear profile $\Phi_{1}$. To simplify notation relabel $\Phi_{1} = u$, and let $v_{s, t}$ be as in $(\ref{4.29})$. Using the same arguments as in the proof of Theorem $\ref{t1.2}$, there exists a sequence $t_{n} \nearrow T^{+}(u)$, $R_{n} \nearrow \infty$, $s_{n} \in \mathbb{R}^{d}$, $\tilde{\lambda}(t) \leq \lambda(t)$, such that
\begin{equation}\label{5.1}
E(\chi(\frac{(x - x(t_{n})) \tilde{\lambda}(t_{n})}{R_{n} \lambda(t_{n})}) v_{s_{n}, t_{n}}) \rightarrow 0,
\end{equation}
\begin{equation}\label{5.2}
\| (1 - \chi(\frac{(x - x(t_{n})) \tilde{\lambda}(t_{n})}{R_{n} \lambda(t_{n})})) v_{s_{n}, t_{n}} \|_{L^{2}} \rightarrow 0,
\end{equation}
\begin{equation}
\| v_{s_{n}, t_{n}} \|_{L^{2}} \nearrow \| u \|_{L^{2}},
\end{equation}
and
\begin{equation}\label{5.3}
\| \chi(\frac{(x - x(t_{n})) \tilde{\lambda}(t_{n})}{R_{n} \lambda(t_{n})}) I v_{s_{n}, t_{n}} \|_{L^{2 + \frac{4}{d}}} \sim 1.
\end{equation}
Therefore, by the almost periodicity of $v$, there exists a sequence $g(t_{n})$ given by $(\ref{2.1.1})$ such that
\begin{equation}\label{5.4}
g(t_{n}) v(t_{n}) \rightarrow u_{0}, \qquad \text{in} \qquad L^{2},
\end{equation}
where $E(u_{0}) = 0$ and $\| u_{0} \|_{L^{2}} \geq \| Q \|_{L^{2}}$.\medskip

Next, utilize a blowup result of \cite{merle2003sharp}, \cite{merle2004universality}, \cite{merle2005blow}, \cite{merle2006sharp}. We will state it here as it is stated in Theorem $3$ of \cite{merle2006sharp}. See also Theorem $3.1$ of \cite{fan20182}.
\begin{theorem}\label{t5.2}
Assume $u$ is a solution to $(\ref{1.1})$ with $H^{1}$ initial data, non-positive energy, and satisfies $(\ref{1.18})$. If $u$ is of zero energy, then $u$ blows up in finite time according to the log-log law,
\begin{equation}\label{5.5}
u(t,x) = \frac{1}{\lambda(t)^{d/2}} (Q + \epsilon)(\frac{x - x(t)}{\lambda(t)}) e^{i \gamma(t)}, \qquad x(t) \in \mathbb{R}^{d}, \qquad \gamma(t) \in \mathbb{R}, \lambda(t) > 0, \qquad \| \epsilon \|_{H^{1}} \leq \delta(\alpha),
\end{equation}
with the estimate
\begin{equation}\label{5.6}
\lambda(t) \sim \sqrt{\frac{T - t}{\ln|\ln(T - t)|}},
\end{equation}
and
\begin{equation}\label{5.7}
\lim_{t \rightarrow T} \int (|\nabla \epsilon(t,x)|^{2} + |\epsilon(t,x)|^{2} e^{-|x|}) dx = 0.
\end{equation}
\end{theorem}

Let $u$ be the solution to $(\ref{1.1})$ with initial data $u_{0}$. If $\| u_{0} \|_{L^{2}} = \| Q \|_{L^{2}}$ then we are done, using the analysis in the previous section. If $\| u_{0} \|_{L^{2}} > \| Q \|_{L^{2}}$, then Theorem $\ref{t5.2}$ implies that $u$ must be of the form $(\ref{5.5})$. Furthermore, by perturbative arguments, for any fixed $t' \in \mathbb{R}$, $(\ref{5.4})$ implies that there exists a sequence $g(t_{n}, t')$ such that
\begin{equation}\label{5.8}
g(t_{n}, t') v(t_{n} + \frac{t'}{\lambda(t_{n})^{2}}) \rightarrow u(t'), \qquad \text{in} \qquad L^{2}.
\end{equation}
In fact, perturbative arguments also imply that there exists a sequence $t_{n}' \nearrow \infty$, perhaps very slowly, such that
\begin{equation}\label{5.9}
\| g(t_{n}, t_{n}') v(t_{n} + \frac{t_{n}'}{\lambda(t_{n})^{2}}) - u(t_{n}') \|_{L^{2}} \rightarrow 0.
\end{equation}
Furthermore, Theorem $\ref{t5.2}$ implies that there exists a sequence $g(t_{n}')$ such that
\begin{equation}\label{5.10}
g(t_{n}') u(t_{n}') \rightharpoonup Q, \qquad \text{weakly in} \qquad L^{2}.
\end{equation}
Combining $(\ref{5.9})$ and $(\ref{5.10})$,
\begin{equation}\label{5.11}
g(t_{n}') g(t_{n}, t_{n}') v(t_{n} + \frac{t_{n}'}{\lambda(t_{n})^{2}}) \rightharpoonup Q, \qquad \text{weakly in} \qquad L^{2}.
\end{equation}
This completes the proof of Theorem $\ref{t1.3}$.\medskip

Acknowledgements: During the writing of this paper, the author was supported by NSF grant DMS - 1764358.

\bibliography{biblio}
\bibliographystyle{plain}

\end{document}